\documentclass[36pt, a4paper]{article}
\usepackage{amsthm,amsmath,amsfonts,amssymb, hyperref,authblk}
\usepackage{blindtext}
\usepackage{lipsum}                     
\usepackage{xargs}                      
\usepackage[pdftex,dvipsnames]{xcolor}  
\usepackage[colorinlistoftodos,prependcaption,textsize=tiny]{todonotes}
\newcommandx{\unsure}[2][1=]{\todo[linecolor=red,backgroundcolor=red!25,bordercolor=red,#1]{#2}}
\newcommandx{\change}[2][1=]{\todo[linecolor=blue,backgroundcolor=blue!25,bordercolor=blue,#1]{#2}}
\newcommandx{\info}[2][1=]{\todo[linecolor=OliveGreen,backgroundcolor=OliveGreen!25,bordercolor=OliveGreen,#1]{#2}}
\newcommandx{\improvement}[2][1=]{\todo[linecolor=Plum,backgroundcolor=Plum!25,bordercolor=Plum,#1]{#2}}
\newcommandx{\thiswillnotshow}[2][1=]{\todo[disable,#1]{#2}}
\newenvironment{itemizeminus}{\begin{itemize} }{\end{itemize}}
\newcommand{\assign}{:=}

\newcommand{\tmop}[1]{\ensuremath{\operatorname{#1}}}
\usepackage{xcolor}
\usepackage{bm}
\usepackage{cases}
\usepackage{booktabs} 
\usepackage{array}
\usepackage{siunitx}
\usepackage{scrtime}
\usepackage{mwe} 
\usepackage{caption}
\usepackage{graphics}
\usepackage{caption}
\usepackage{float}
\usepackage{stmaryrd}
\usepackage{xparse}
\usepackage{booktabs}
\usepackage[noend]{algorithmic} 
\usepackage{algorithm}
\usepackage{stmaryrd}
\usepackage[figuresright]{rotating}

\usepackage{vmargin}
\setmarginsrb{4cm}{1cm}{2cm}{1cm}{1cm}{1cm}{2cm}{4cm}

\renewcommand{\geq}{\geqslant}

\newcolumntype{L}[1]{>{\raggedright\arraybackslash}p{#1}}

\everymath{\displaystyle}

\newtheorem{thm}{Theorem}[section]
\newtheorem{lemma}[thm]{Lemma}
\newtheorem{proposition}[thm]{Proposition}

\newtheorem{e-definition}[thm]{Definition}

\numberwithin{equation}{section}

\makeatletter
\newcommand\appendix@section[1]{%
\refstepcounter{section}%
\orig@section*{Appendix \@Alph\c@section: #1}%
}
\let\orig@section\section
\g@addto@macro\appendix{\let\section\appendix@section}
\makeatother

 \title{Stability Analysis of a Non-Separable Mean-Field Games for Pedestrian Flow in Large Corridors}     

 \author{Mohamed Ghattassi \thanks{NYUAD Research Institute, New York University Abu Dhabi, PO Box 129188, Abu Dhabi, United Arab Emirates, {\sf mg6888@nyu.edu}} \quad Nader Masmoudi \thanks{Department of Mathematics, New York University in Abu Dhabi, Saadiyat Island, P.O. Box 129188, Abu Dhabi, United Arab Emirates-- Courant Institute of Mathematical Sciences, New York University, 251 Mercer Street, New York, NY 10012, USA, {\sf nm30@nyu.edu}} \quad  Eliot Pacherie \thanks{CNRS and CY Cergy Paris Université, 2 avenue Adolphe Chauvin, 95300 Pontoise, France {\sf eliot.pacherie@cyu.fr}}}

\date{\today}

\begin{document}                  
\maketitle   

\begin{abstract}

We investigate the existence and stability of small perturbations of constant states of the generalized Hughes model for pedestrian flow in an infinitely large corridor. We show that constant flows are stable under a condition on the density. Our findings indicates that when the density is less than half of the maximum density $\rho_{m}/2$, which is the Lasry-Lions monotonicity condition, we can control the perturbation and prove postive stability results for the nonlinear Generalized Hughes model. However, due to wave propagation phenomena, we are unable to provide an answer for stability results when the density is higher. Our approach involves constructing an explicit solution for the linear problem in Fourier analysis and demonstrating, through a fixed-point argument, how to construct the solution for the full nonlinear mean-field games system.

   \vspace{0.2cm}
\end{abstract}

 \tableofcontents

\section{Introduction and presentation of the results}

\subsection{Model for a pedestrian flow in a large corridor}
In recent years, crowd motion modeling has become an active area of research. As the world population and urbanization continue to increase, we cannot underestimate the importance of understanding the behavior of human crowds. Nowadays, there is a lot of interdisciplinary research among sociologists, biologists, physicists, and mathematicians; crowd motion has become one of the emerging research topics.

In this paper, we analyze the stability of the Generalized Hughes model (GH for short), which is a class of non-separable Mean-field games (MFGs for short) systems describing pedestrian flow (see \cite{gm2023}). The GH model for pedestrian dynamics is based on a Kolmogorov-type equation (describing the evolution of crowd density $\rho$) and a Hamilton-Jacobi-Bellman equation (HJB for short) (providing the shortest weighted distance to an exit).
\begin{align}
\partial_t\rho - \nabla \cdot \left( \rho \mathcal{H}_{p}(x,y,\rho,\nabla \phi)\right) -\sigma \Delta  \rho= 0, &\qquad (x,y)\in\Omega, \quad t \in (0,T) \label{GH1}\\ 
-\partial_t\phi + \mathcal{H}(x,y,\rho,\nabla \phi) - \sigma \Delta  \phi=0, &\qquad  (x,y)\in\Omega, \quad t \in (0,T) \label{GH2}\\
\rho(x,y,t=0)= \rho_{0}(x,y),&\quad (x,y)\in\Omega \label{GHi1}\\ 
\phi(x,y,t=T)= \phi_{T}(x,y),&\quad (x,y)\in\Omega \label{GHi2}
\end{align}
where the viscosity $\sigma \geqslant 0$ corresponds to some noise effect, the non-separable Hamiltonian $\mathcal{H}$ is given by 
\begin{align*}
  \mathcal{H}:\mathbb{R} \times\mathbb{R} \times \mathbb{R}\times \mathbb{R}^2 &\longrightarrow \mathbb{R}& \\
 (x,y,\rho,p)&\longmapsto \frac{1}{2} f^{\beta}(\rho)|p|^{2} -\frac{1}{2} f^{2-\beta}(\rho),&
\end{align*}
$(x,y)\in \Omega$ denotes the position in space which can be a corridor of length $2 L > 0$ and height $2 h > 0$ defined by the set $(x, y) \in [- L,\,L] \times [- h,\, h]$, $t\in (0,T]$, $T>0$, the  horizon time, $\nabla$ and $\Delta$ are respectively the gradient and Laplacian with respect to the space variable $x$.  $\mathcal{H}$ is a convex function of the variable $p$, $\mathcal{H}_{p}$ stands for $\frac{\partial\mathcal{H}}{\partial p}(x, y,\rho, p)$. The system \eqref{GH1}-\eqref{GHi2} is a macroscopic model for crowd dynamics, in which individuals seek to minimize their travel time by avoiding/ or traveling to regions of high density. Here, traveling or avoiding regions of high density depends on the factor $\beta\in [0,2]$, as already shown in \cite{gm2023}.
 This model describes evacuation scenarios, where a group of people wants to leave a domain $\Omega \subset\mathbb{R}^{2}$ with one or several exits/doors and/or obstacles as fast as possible. The function $\rho$ corresponds to the pedestrian density, $\rho_{in}$ is the initial pedestrian density and $\phi$ is the weighted shortest distance to the exit. $f$ is a function introducing saturation effects, and we take here $f(\rho) \assign \rho_{m}-\rho$, where $\rho_{m}$ corresponds to the maximum scaled pedestrian density. The MFG system \eqref{GH1}-\eqref{GHi2} is supplemented with different boundary conditions for the
walls and exits.  We assume that the boundary $\partial \Omega$ of our domain is subdivided into three parts: inflow $\Gamma$, outflow $\Sigma$ and insulation $\Gamma_{a}=\partial \Omega / (\Gamma \cup\Sigma)$ with $\Gamma \cap \Sigma =\emptyset$. Here we assume that particles enter the domain $\Omega$ on $\Gamma$ with boundary pedestrian density $\rho_{b}$. On the remaining part of the boundary $\Gamma_{a}$, we impose no flux conditions 
\begin{align}
\rho = 0,&\qquad t \in (0,T), \quad (x,y)\in\Sigma \label{Boundrho1}\\ 
\left(\sigma\nabla\rho +\rho\mathcal{H}_{p}(x,y,\rho,\nabla \phi)\right)\cdot n(x,y)= 0,&\qquad t \in (0,T), \quad (x,y)\in\Gamma_{a} \label{Boundrho2}\\ 
 \rho=\rho_{b}, &\qquad t \in (0,T), \quad (x,y)\in\Gamma, \label{Boundar3}
\end{align}
where  $n(x,y)$ denotes the outer normal vector to the boundary, for the HJB-type equation we consider the following boundary conditions 
\begin{align}
\phi = 0,&\qquad t \in (0,T), \quad (x,y)\in\Sigma \label{Boundrho1}\\ 
\nabla \phi\cdot n(x,y)= 0,&\qquad t \in (0,T), \quad (x,y)\in\Gamma_{a} \label{Boundrho2}\\ 
 \phi=\phi_{b}, &\qquad t \in (0,T), \quad (x,y)\in\Gamma, \label{Boundar3}
\end{align}
  where the function $\phi$ takes zeros on the exit $\Sigma$. The MFG presentation of the pedestrian flow \eqref{GH1}-\eqref{GHi2} shown in \cite{gm2023}, is based on the pioneering work of R.L.Hughes, \cite{hughes2002continuum}. In \cite{hughes2002continuum}, the author introduces the macroscopic discerption treating the crowd dynamic as a fluid made of particles.  The model is given by a nonlinear conservation law with discontinuous flux, coupled with an Eikonal equation via a cost functional. The existence and uniqueness analysis for the Hughes model appears to be challenging.  However, we stress that for the unique viscosity solution of the eikonal equation, no more than Lipschitz continuity can be expected. These various difficulties motivated the development of several attempts to study, both analytically and numerically \cite{burger2014mean,carrillo2016improved,bruna2017cross}, the Hughes model and its regularized variants, and for a more detailed description of the model, we refer the reader to the recent survey on that topic \cite{amadori2023mathematical} and references therein. Recently, in \cite{gm2023}, the authors propose a new generalized Hughes model which is an intelligent fluid model  describing the congestion where pedestrian group are attracted by the high-density region or avoids . The idea is to introduce a new factor $\beta \in [0,2]$ into the Hughes model. With this new control parameter, people can either avoid or move towards high-density areas. Then, the authors reformulated the model within the framework of non-separable mean-field games (MFGs). They demonstrated the existence of weak solutions and analyzed the vanishing viscosity limit of these weak solutions. As previously mentioned in \cite{gm2023}, Hughes' derivation is rooted in the assumption that pedestrians cannot predict the population's evolution but choose their strategies solely based on the information available at time $t$, without anticipation. This same argument was employed to justify the use of the quasi-stationary Hamilton-Jacobi-Bellman equation, as discussed in \cite{mouzouni2020quasi} and \cite{camilli2023quasi}. For more comprehensive insights into the existence and uniqueness of solutions for the MFGs, we refer the reader to the course by P.L. Lions \cite{Lions2023}. For the existence of weak solutions,  see \cite{porretta2015weak} and \cite{achdou2018mean}, while for the existence of strong solutions, we can refer to \cite{gomes2015time0}, \cite{gomes2016time}, \cite{ambrose2018strong} and references therein.

 The stability of MFGs has been explored in several articles. It has been established that a monotonicity condition on the cost functions leads to uniqueness and stability of solutions, as well as the turnpike property and the convergence of solutions towards the stationary ergodic state. This was first proved in \cite{cardaliaguet2012long,cardaliaguet2013long} for quadratic separable Hamiltonian systems and in \cite{gomes2010discrete} for discrete-time finite-state systems. \cite{cardaliaguet2019long} and \cite{porretta2018turnpike} further discussed the long-term behavior in the case of smoothing couplings and uniformly convex Hamiltonian, and local couplings and globally Lipschitz Hamiltonian, respectively. \cite{cesaroni2021brake,cirant2019existence,cirant2018variational} examined the stable long-time behavior of MFGs with some monotonicity structure. \cite{cirant2021long} showed that the monotonicity of the cost functions can be relaxed by taking into account the diffusive character of the equations. Recently, \cite{mimikos2022regularity} removed the coercivity assumption used in \cite{munoz2022classical,munoz2023classical} for the one-dimensional case and characterized the long-time behavior of the solutions, showing that they satisfy the turnpike property with an exponential rate of convergence and that they converge to the solution of the infinite-horizon system. More recently, \cite{cesaroni2023stationary} provided some partial answers to the stability of the Kuramoto MFG system with strong interaction \cite{carmona2022synchronization}.

In the present paper, we will consider a non-separable MFGs for pedestrian flow addressing both scenarios: when pedestrians avoid high-density areas ($\beta=2$) and when pedestrians are attracted by high-density areas ($\beta=0$). However, the analysis remains applicable for $\beta\in (0,2)$. We begin by considering the second order MFGs system \eqref{GH1}-\eqref{GHi2} for $\beta=2$ as follows 
\begin{eqnarray}
  \partial_t \rho - \nabla . (\rho f^2 (\rho) \nabla \phi) - \sigma \Delta
  \rho & = & 0  \qquad (x,y)\in\Omega, \quad t \in (0,T)\label{maineq}\\
  \partial_t \phi - \frac{1}{2} f^2 (\rho) | \nabla \phi |^2 + \sigma
  \Delta \phi + \frac{1}{2} & = & 0 \qquad (x,y)\in\Omega, \quad t \in (0,T) \label{maineq2}\\
  \rho(x,y,t=0)&=& \rho_{0}(x,y),\quad (x,y)\in\Omega \label{GHi12}\\ 
\phi(x,y,t=T)&=& \phi_{T}(x,y),\quad (x,y)\in\Omega. \label{GHi22}
\end{eqnarray} 
 A particular solution of this problem is
\begin{equation}\label{ParSol}
 \rho (x, y, t) = \bar{\rho}, \,\,\phi (x, y, t) = \frac{x}{f (\bar{\rho})},
\end{equation}
for some constant $\rho_m > \bar{\rho} > 0$. This corresponds to a constant flow of pedestrians, all moving at the same speed. We are interested in the stability of this equation. That is, if our initial data is close to $\left( \bar{\rho}, \frac{x}{f (\bar{\rho})} \right)$, does this stay true for all time between $0$ and $T$ ? And does this stay true if $\sigma = 0$, that is without any viscosity effect? This poses many difficulties, as even the question of the well-posedness of such a problem is not understood. This is in part because the second equation is backward in time.The other case in which pedestrians are attracted by the high-density ($\beta =0$) which may generate some concentration phenomena will also be discussed. This corresponds to the equation
\begin{eqnarray}
  \partial_t \rho - \nabla . (\rho \nabla \phi) - \sigma \Delta
  \rho & = & 0  \qquad  (x,y)\in\Omega, \quad t \in (0,T)\label{maineq0}\\
  \partial_t \phi - \frac{1}{2} | \nabla \phi |^2 + \sigma
  \Delta \phi + \frac{1}{2} f^2 (\rho) & = & 0 \qquad (x,y)\in\Omega, \quad t \in (0,T)  \label{maineq20}\\
    \rho(x,y,t=0)&=& \rho_{0}(x,y),\quad (x,y)\in\Omega \label{GHi10}\\ 
\phi(x,y,t=T)&=& \phi_{T}(x,y),\quad (x,y)\in\Omega \label{GHi20}
\end{eqnarray}

In this paper, we want to answer these questions for infinitely large corridors, that is, $L = h = + \infty$. In other words, we look at the system on $\mathbb{R}^2$ rather than a rectangle. In that case, there are no more boundary conditions, which will simplify the problem, while still keeping the main difficulty, which is the forward-backward nature, because evolution runs forward while optimization runs backward by the Bellman dynamic programming principle. Our main result will be this stability, with and without viscosity, under the condition that $\bar{\rho} < \frac{\rho_m}{2}$ which is going with the Lasry-Lions monotonicity condition proved in \cite{gm2023}, see Theorem \ref{thmStab} below. We expect the stability to fail if $\bar{\rho} > \frac{\rho_m}{2}$, and we will discuss this later on.

Let us now write the problem we are considering for $\beta = 2$. We look for a solution of
equations (\ref{maineq})-(\ref{maineq2}) of the form
\[ \left( \bar{\rho} + \psi (x, y, t), \frac{x}{f (\bar{\rho})} + \varphi (x,
   y, t) \right) \]
where $\psi, \varphi: \mathbb{R}^2 \times [0, T] \rightarrow \mathbb{R}$. After computations, we check that $\psi,
\varphi$ satisfies the system
\begin{eqnarray}
  \partial_t \psi & = & (f (\bar{\rho}) - 2 \bar{\rho}) \partial_x \psi +
  \bar{\rho} f^2 (\bar{\rho}) \Delta \varphi + \sigma \Delta \psi  \nonumber\\ & + & \nabla . ((\psi f^2 (\bar{\rho}) - 2 f (\bar{\rho}) (\bar{\rho} +\psi) \psi + (\bar{\rho} + \psi) \psi^2) \nabla \varphi) \nonumber \\ &+& \frac{1}{f (\bar{\rho})} \partial_x (\psi^3) +  \left( \frac{\bar{\rho}}{f (\bar{\rho})} - 2 \right) \partial_x
  (\psi^2)  \quad (x,y)\in\mathbb{R}^2 \,\, t\in (0,T),\label{mq1}
\end{eqnarray}
and
\begin{eqnarray}
  \partial_t \varphi & = & f (\bar{\rho}) \partial_x \varphi - \frac{1}{f
  (\bar{\rho})} \psi - \sigma \Delta \varphi \nonumber\\
  & + &  \frac{1}{2} f^2 (\bar{\rho}) | \nabla \varphi |^2 + \frac{1}{2}
  (\psi^2 - 2 f (\bar{\rho}) \psi) | \nabla \varphi |^2 \nonumber\\
  & + & \frac{1}{2 f (\bar{\rho})} \psi^2 + \partial_x \varphi \left(
  \frac{1}{f (\bar{\rho})} \psi^2 - 2 \psi \right)  \quad (x,y)\in\mathbb{R}^2 \,\, t\in (0,T).  \label{mq2}
\end{eqnarray}
In these two equations, we have regrouped the linear terms in the first line.
To the system  \eqref{mq1}-\eqref{mq2}, we associate the initial data 
\begin{equation}\label{inData}
\psi (x, y, 0) = \psi_0 (x,y)\,\,\ \text{and}\,\, \varphi (x, y, T) = \varphi_T (x, y)\quad \text{for} \,\,\, (x, y) \in \mathbb{R}^2.
\end{equation}
Our goal is to show that the system \eqref{mq1}-\eqref{inData} admits a unique solution and that this
solution stays small if we suppose that $\psi_0$ and $\varphi_T$ are small in some sense.

\subsection{Stability results without viscosity}
\subsubsection{Linear stability analysis without viscosity}
Here, we look at the linearized version of the above problem without viscosity, that is the system of equations
\begin{equation}\label{losttimes}
  \begin{cases}
    \partial_t \psi = (f (\bar{\rho}) - 2 \bar{\rho}) \partial_x \psi + \bar{\rho} f^2 (\bar{\rho}) \Delta \varphi & \qquad  (x,y)\in\mathbb{R}^2,\,\,\,\, t\in (0,T)\\
    \partial_t \varphi = f (\bar{\rho}) \partial_x \varphi - \frac{1}{f(\bar{\rho})} \psi &\qquad  (x,y)\in\mathbb{R}^2, \,\,\,\, t\in (0,T)\\
    \psi (x, y, 0) = \psi_0 (x, y)  &\qquad  (x,y)\in\mathbb{R}^2 \\
    \varphi (x,y, T) = \varphi_T (x, y)&\qquad  (x,y)\in\mathbb{R}^2.
  \end{cases}
\end{equation}
Now, we will state our main linear stability result for system \eqref{losttimes} as follow.
\begin{proposition}
  \label{nina}Suppose that $\bar{\rho} < \frac{\rho_m}{2}$. Then, there exists
  $K > 0$ such that for any $T > 0$, system \eqref{losttimes} with

 \[ \hat{\psi}_0, \widehat{\nabla \varphi_T} \in (L^1 \cap L^{\infty}) (\mathbb{R}^2), \]
  admits a unique solution on $[0, T]$. Moreover, the solution satisfies
  
  \begin{equation}\label{EstLinL2}
  \begin{aligned}
  \| \psi(t)\|_{L^2} + \| \nabla \varphi(t) \|_{L^2} &\leqslant K \left( \frac{\|
     \hat{\psi}_0 (\xi) \|_{L^1 \cap L^{\infty}}}{(1 + t)} + \frac{\| | \xi |
     \hat{\varphi}_T (\xi) \|_{L^1 \cap L^{\infty}}}{(1 + T - t)} \right)
   \end{aligned}
   \end{equation}
and
  \begin{equation}\label{EstLinLinf}
  \begin{aligned}
 \| \psi(t)\|_{L^{\infty}} + \| \nabla \varphi(t) \|_{L^{\infty}} &\leqslant K
   \left( \frac{\| \hat{\psi}_0 (\xi) \|_{L^1 \cap L^{\infty}}}{(1 + t)^2} +
   \frac{\| | \xi | \hat{\varphi}_T (\xi) \|_{L^1 \cap L^{\infty}}}{(1 + T -
   t)^2} \right) . 
   \end{aligned}
   \end{equation}
   \end{proposition}
 
 Note the uniformity of this result with respect to $T$. The initial perturbation occurs at $t=0$ for $\psi$ and at $t=T$ for $\varphi$. Consequently, we anticipate that for large values of $T$, when we estimate the error at a time $t$ much smaller than $T$, only the impact of the perturbation in $\psi$  will have a discernible effect. This observation indeed follows as a consequence of the aforementioned result.

\

Let us explain the proof of this result. Taking the Fourier transform of this
system yields, with $\xi = (\xi_1, \xi_2) \in \mathbb{R}^2$ the Fourier
coordinates,
\[ \partial_t \left(\begin{array}{c}
     \hat{\psi}\\
     \hat{\varphi}
   \end{array}\right) = A (\xi) \left(\begin{array}{c}
     \hat{\psi}\\
     \hat{\varphi}
   \end{array}\right) \]
where
\begin{equation}
  A (\xi) \assign \left(\begin{array}{cc}
    i (f (\bar{\rho}) - 2 \bar{\rho}) \xi_1 & - \bar{\rho} f^2 (\bar{\rho}) |
    \xi |^2\\
    \frac{- 1}{f (\bar{\rho})} & i f (\bar{\rho}) \xi_1
  \end{array}\right) . \label{matA}
\end{equation}
We introduce the quantity
\[ \theta (\xi) \assign 2 \sqrt{f (\bar{\rho}) \bar{\rho} | \xi |^2 -
   \bar{\rho}^2 | \xi_1 |^2}, \]
and the checkthat the eigenvalues of $A (\xi)$ are
\begin{eqnarray*}
  \lambda_1 (\xi) & \assign & \frac{1}{2} (- \theta (\xi) + 2 i (f
  (\bar{\rho}) - \bar{\rho}) \xi_1)\\
  \lambda_2 (\xi) & \assign & \frac{1}{2} (\theta (\xi) + 2 i (f (\bar{\rho})
  - \bar{\rho}) \xi_1) .
\end{eqnarray*}
To get some decay, it will be important that $\lambda_1$ and $\lambda_2$ have
nonzero real parts for all $\xi \neq 0$. Here the condition
$\bar{\rho} < \frac{\rho_m}{2}$ appears: we need what is below the square root
in $\theta (\xi)$ to be positive for all $\xi \neq 0$, and we check that
\begin{equation}\label{thetaEq}
 \theta (\xi) = 2 \sqrt{\bar{\rho} (f (\bar{\rho}) - \bar{\rho}) | \xi_1 |^2
   + \bar{\rho} f (\bar{\rho}) | \xi_2 |^2}, 
   \end{equation}
and thus we need $f (\bar{\rho}) - \bar{\rho} > 0$, which is equivalent to
$\bar{\rho} < \frac{\rho_m}{2}$. Then, we will diagonalize the matrix $A$, and
integrate the coordinates forward in time one, and backward in time the
second. This will fix $\varphi_0 (x, y) = \varphi (x, y, 0)$ and $\psi_T (x,
y) = \psi (x, y, T)$ given the initial data $\psi_0, \varphi_T$ through
compatibility conditions, so that we have a decay in time of the solutions. See
Section \ref{lv} for the full proof of Proposition \ref{nina}.
\subsubsection{Nonlinear stability analysis  without viscosity}
We now state the stability under the condition $\bar{\rho} < \frac{\rho_m}{2}$ for the full nonlinear system.

\begin{thm}\label{thmStab}
  \label{aiaiaia}Suppose that $\bar{\rho} < \frac{\rho_m}{2}$. Then, there exists $\varepsilon_0, K > 0$ such that for any $T > 0$, system \eqref{mq1}-\eqref{inData} with
  \[ \| (1 + | \xi |)^3 \hat{\psi}_0 \|_{L^{\infty} (\mathbb{R}^2)} + \| (1 +
     | \xi |)^3 \widehat{\nabla \varphi_T} \|_{L^{\infty} (\mathbb{R}^2)}
     \leqslant \varepsilon_0 \]
  admits a unique solution in $L^2 \cap L^{\infty}$ on $[0, T]$, analytic on $] 0, T
  [$, and this solution satisfies for all $t \in [0, T]$ that
  \begin{equation}\label{NL2}
  \begin{aligned}
     (\| \psi \|_{L^2} + \| \nabla \varphi \|_{L^2}) (t)\leqslant & K \left( \frac{\| (1 + | \xi |)^3 \hat{\psi}_0\|_{L^{\infty} (\mathbb{R}^2)}}{(1 + t)} + \frac{\| (1 + | \xi |)^3 \widehat{\nabla \varphi_T} \|_{L^{\infty} (\mathbb{R}^2)}}{(1 + T - t)}\right),
    \end{aligned}
\end{equation}
  and
  \begin{equation}\label{NLinf}
  \begin{aligned}
     (\| \psi \|_{L^{\infty}} + \| \nabla \varphi \|_{L^{\infty}}) (t) & \leqslant  K \left( \frac{\| (1 + | \xi |)^3 \hat{\psi}_0 \|_{L^{\infty} (\mathbb{R}^2)}}{(1 + t)^2} + \frac{\| (1 + | \xi |)^3 \widehat{\nabla \varphi_T} \|_{L^{\infty} (\mathbb{R}^2)}}{(1 + T - t)^2} \right) .
    \end{aligned}
\end{equation}
\end{thm}

We believe that these decay rates in time are optimal (at least without additional
assumptions on $\psi$ and $\nabla \varphi$). We also show a regularizing
effect: the solution is analytic on $] 0, T [$.

\
We should note that even the existence of a solution to the systems \eqref{mq1}-\eqref{inData} was open. The existence will come with the estimate, as the solution will be constructed and estimated by a fixed point argument. That is, we will construct a first order approximation of the solution with the linear problem. Then, solving the linear problem with a source term, we construct a second order approximation
This iterative process will be repeated infinitely many times, ultimately leading to the creation of an exact solution to the problem. See section \ref{alive} for the proof of this result.

\subsection{Stability analysis with a viscosity term}

We now turn our attention to the scenario where $\sigma > 0$. The introduction of this additional viscosity leads us to anticipate improved stability when $\bar{\rho} < \frac{\rho_m}{2}$. However, this expectation does not appear to hold. It's noteworthy that, as stated in Theorem \ref{thmStab}, the time decay rates are already superior to those of the classical heat equation.  In fact, they are the ones of a fractional heat equation (that is replacing $\Delta$ by $\Delta^{1/2}$).

In this scenario, we are also interested in the case of vanishing viscosity (such as $\sigma \rightarrow 0$), since a small viscosity is commonly used in numerical simulations. Consequently, we will demonstrate the preservation properties in the context of the vanishing viscosity limit for the nonlinear problem. Our result is as follows:

\begin{thm}
  \label{voices}Suppose that $\bar{\rho} < \frac{\rho_m}{2}$. Then, for $\sigma
  > 0$, there exists $K, \varepsilon_0 > 0$ such that system \eqref{mq1}-\eqref{inData} with
  \[ \| (1 + | \xi |)^3(1+\sigma |\xi|) \hat{\psi}_0 \|_{L^{\infty} (\mathbb{R}^2)} + \| (1 +
     | \xi |)^3 (| \xi | + \sigma | \xi |^2) \widehat{\varphi_T}
     \|_{L^{\infty} (\mathbb{R}^2)} \leqslant \varepsilon_0 \]
  admits a unique solution ($\psi_{\sigma}, \varphi_{\sigma}$) in $L^2 \cap
  L^{\infty}$ on $[0, T]$ and this solution satisfies for all $t \in [0, T]$
  that
  \begin{equation}\label{viscoNL2}
  \begin{aligned}
 (\| \psi_{\sigma} \|_{L^2} + \| \nabla \varphi_{\sigma} \|_{L^2} + & \sigma \| \nabla^2 \varphi_{\sigma} \|_{L^2}) (t) \\& \leqslant  K \left( \frac{\| (1 + | \xi |)^3(1+\sigma|\xi|) \hat{\psi}_0
    \|_{L^{\infty} (\mathbb{R}^2)}}{(1 + t)} + \frac{\| (1 + | \xi |)^3 (| \xi | + \sigma | \xi |^2) \widehat{\varphi_T} \|_{L^{\infty}(\mathbb{R}^2)}}{(1 + T - t)} \right)
    \end{aligned}
\end{equation}
  and
  \begin{equation}\label{viscoNLinf}
  \begin{aligned}
 (\| \psi_{\sigma} \|_{L^{\infty}} + \| \nabla \varphi_{\sigma}\|_{L^{\infty}} &+ \sigma \| \nabla^2 \varphi_{\sigma} \|_{L^{\infty}})
    (t) \\& \leqslant  K \left( \frac{\| (1 + | \xi |)^3(1+\sigma|\xi|) \hat{\psi}_0 \|_{L^{\infty} (\mathbb{R}^2)}}{(1 + t)^2} + \frac{\| (1 + | \xi |)^3 (|\xi | + \sigma | \xi |^2) \widehat{\varphi_T} \|_{L^{\infty}(\mathbb{R}^2)}}{(1 + T - t)^2} \right).
    \end{aligned}
  \end{equation}
  Furthermore, if we consider ($\psi_{\sigma}, \varphi_{\sigma}$) the solution to the system  \eqref{mq1}-\eqref{inData} with viscosity and $(\psi, \varphi)$ the solution for the vanishing viscosity system  \eqref{mq1}-\eqref{inData} with the same initial data $\psi_0, \varphi_T$
  and
  \[ \| (1 + | \xi |)^5 (1+\sigma|\xi|))\hat{\psi}_0 \|_{L^{\infty} (\mathbb{R}^2)} + {\| (1 +
     | \xi |)^7 \widehat{\varphi_T} \|_{L^{\infty} (\mathbb{R}^2)} < + \infty}
  \]
  then
  \[ \| \psi_{\sigma} - \psi \|_{L^2 \cap L^{\infty}} + \| \nabla
     \varphi_{\sigma} - \nabla \varphi \|_{L^2 \cap L^{\infty}} \rightarrow 0,  \,\,\,\text{as}\,\,\, \sigma \rightarrow 0,
  \]
  uniformly on $[0, T]$.
\end{thm}

We see that in the second part of this result, we have a loss of derivative, in the sense that we have the convergence in $L^2 \cap L^{\infty}$ but the initial perturbation is in a space of smoother functions.

\subsection{Some remarks}

\begin{itemizeminus}

\item Based on the stability results provided in the absence of viscosity \eqref{EstLinL2}-\eqref{EstLinLinf}, the nonlinear stability analysis without viscosity \eqref{NL2}\eqref{NLinf}, and the nonlinear stability analysis with viscosity \eqref{viscoNL2}-\eqref{viscoNLinf}, we observe that when the time horizon $T=+\infty$ there is no requirement to impose initial data for the value function at $T=+\infty$. In addition, the perturbations are controlled by the initial perturbation in the density. This implies that for a sufficiently large time horizon $T$, there is no need to specify the initial value of the potential, because the solution will attain an equilibrium state after some time $T>0$.

\item The analysis presented here is still applicable even when we consider the MFG system with a prescribed terminal condition as follows $\rho(t=0,x,y)=\rho_{0}(x,y)$ and $ \phi(t=T,x,y)=G(x,y,m(T))$ as well as to the so-called planning problem with a prescribed terminal density $\rho(t=0,x,y)=\rho_{0}(x,y)$ and $\rho(t=T,x,y)=\rho_{T}(x,y)$, see \cite{mimikos2022regularity}.

\item In the case $\bar{\rho} > \frac{\rho_m}{2}$ (still with $\beta = 2$), the key difference is that now there exist frequencies $\xi \neq 0$ such that $\mathfrak{R}\mathfrak{e}
(\theta (\xi)) = 0$. This implies, for the linear problem, the existence of
planar wave solutions that do not decay in time.

\begin{lemma}
  \label{vortex}Consider the linear problem without viscosity
\begin{equation*}
  \begin{cases}
     \partial_t \psi - (f (\bar{\rho}) - 2 \bar{\rho}) \partial_x \psi - \bar{\rho} f^2 (\bar{\rho}) \Delta \varphi  =0 &\quad (x,y)\in\mathbb{R}^2 \,\,\,\, t\in (0,T)\\
    \partial_t \varphi - f (\bar{\rho}) \partial_x \varphi +\frac{1}{f(\bar{\rho})} \psi =0 &\quad  (x,y)\in\mathbb{R}^2 \,\,\,\, t\in (0,T)\\
    \psi (x, y, 0) = \psi_0 (x, y)  &\qquad  (x,y)\in\mathbb{R}^2 \\
    \varphi (x,y, T) = \varphi_T (x, y)&\qquad  (x,y)\in\mathbb{R}^2.
  \end{cases}
\end{equation*}
Then, if $\bar{\rho} > \frac{\rho_m}{2}$, this problem admits planar waves solution of the form
  \[ \psi = A \cos (a x + b y + c t), \]
  \[ \varphi = B \sin (a x + b y + c t) \]
  where $a, b, c \in \mathbb{R}$ and $A B \neq 0$.
\end{lemma}

Such solutions of the linear problem also exist if we add a small viscosity.
These objects do not exist if $\bar{\rho} < \frac{\rho_m}{2}$. By themselves,
these solutions do not prevent solutions of the linear problem with finite
energy to decay, since they have infinite energy. We are not able at this
point to say whether or not small solutions of the equation with $\bar{\rho} >
\frac{\rho_m}{2}$ and finite energy are stable or not. What is clear, however,
is that we lose the regularizing effect that happened when $\bar{\rho} <
\frac{\rho_m}{2}$, and thus we expect either instability or at least a weaker
stability result (for instance, the decay in time will be slower than in
Theorem \ref{aiaiaia}). Adding viscosity does not clearly implies stability
here.

Now concerning the case $\beta = 0$ (see system \eqref{maineq0}-\eqref{GHi20}), the situation is even
worse. The stationary solution is $(\bar{\rho}, x f (\bar{\rho}))$, and the
linear problem (for $\sigma = 0$) becomes
   \begin{equation*}
  \begin{cases}
    \partial_t \psi - f (\bar{\rho}) \partial_{x} \psi - \bar{\rho} \Delta\varphi = 0 &\quad  (x,y)\in\mathbb{R}^2, \,\,\,\, t\in (0,T)\\
   \partial_t \varphi - f (\bar{\rho}) \partial_{x} \varphi - f(\bar{\rho}) \psi =0 &\quad  (x,y)\in\mathbb{R}^2, \,\,\,\, t\in (0,T)\\
    \psi (x, y, 0) = \psi_0 (x, y)  &\qquad  (x,y)\in\mathbb{R}^2 \\
    \varphi (x,y, T) = \varphi_T (x, y)&\qquad  (x,y)\in\mathbb{R}^2.
  \end{cases}
\end{equation*}

Now, this problem also admits standing waves, but for all values of $a$ and $b$ (and independently of $\bar{\rho}$).

\begin{lemma}
  \label{light}For any $0 < \bar{\rho} < \rho_m$ and any $a, b \in
  \mathbb{R}$, there exists $A, B \in \mathbb{R}$ with $A B \neq 0$ and $c \in
  \mathbb{R}$ such that the equation aboves admits a solution of the form
  \[ \left\{\begin{array}{l}
       \psi = A \cos (a x + b y + c t)\\
       \varphi = B \sin (a x + b y + c t) .
     \end{array}\right. \]
\end{lemma}
\end{itemizeminus}
\subsection{Open problems and plan of the paper}

Let us state here a list of open problems related to this result.
\begin{itemizeminus}
  \item Can we show stability or instability in the case $\bar{\rho} >
  \frac{\rho_m}{2}$, with or without viscosity, and uniformly or not in the
  time $T > 0$ ? At least at the linear level, some frequencies can be treated
as in the case $\bar{\rho} < \frac{\rho_m}{2}$, but others seem to follow more
of a wave-type equation. Some frequencies, at the boundary of both previous cases, do not decay (see Lemma \ref{vortex}), but they are rare. The
  nonlinear problem introduce the additional difficulty of mixing the
  frequencies.
  
  
  \item The problem of corridors of finite size $(h < + \infty)$ should work with a similar argument. Most of the proofs should hold if we replace $\xi_2 \in
  \mathbb{R}$ by a sequence of frequencies. We also expect that in the case $\bar{\rho} < \frac{\rho_m}{2}$ we have stability. Having a finite length $(L < + \infty)$ should only help, as perturbations will be evacuated from the corridor after some time.

  \item Our method, of having both the stability and the existence for a
  perturbation of a stationary solution by a fixed point argument, could be
  applied to other MFGs problem written as a forward/backward
  problem. Indeed, let's examine the following second-order MFGs
\begin{align}\label{systemMFGsep}
\partial_t \rho  - \nabla \cdot \left( \rho H_{p}(t,x,y,\nabla  \phi )\right)  - \sigma \Delta  \rho = 0, &\\
-\partial_t \phi  + H(t,x,y,\nabla  \phi ) - \sigma \Delta   \phi  =F(t,x,y,\rho), &
\end{align}
where $t \in(0,T)$, for horizon time $T>0$,  $(x,y)\in \mathbb{R}^{2}$,  $H$ is a convex function of $p$, $ H_{p}$ stands for $\frac{ \partial\mathcal{H}}{\partial p}$ and $F$ is a coupling term. Let us consider the following assumptions on the Hamiltonian $H$ and on the function $F$
\begin{equation}\label{systemMFGsep1}
H(t,x,y,\nabla  \phi ) = \frac{1}{2} |\nabla  \phi |^{2},\,\,\,\quad F(t,x,y,m)= \rho^{\beta} ,\,\,\, \text{for} \,\, \beta> 0,
\end{equation}
for more details we refer to \cite{lasry2007mean} and references therein. A particular solution $\left(\bar{m},\bar{ u}\right)=\left(\bar{m},\sqrt{2\bar{m}^{\beta}}x\right)$  with $\bar{m}$ is a positive constant, can be constructed for system \eqref{systemMFGsep}-\eqref{systemMFGsep1}. Nevertheless, our analysis can address the stability analysis of the system defined by equations \eqref{systemMFGsep}-\eqref{systemMFGsep1}. However, extending these findings to MFGs with strong aggregation, as seen in \cite{cirant2021maximal, cirant2022existence}, and the mean-field game master equations introduced by Lions in \cite{Lions2023} and also discussed in \cite{cardaliaguet2019master}, remains a challenging question.
%

\end{itemizeminus}
The rest of this paper is structured as follows. Section \ref{lv} is devoted to
the proof of Proposition \ref{nina}, section \ref{alive} to the proof of
Theorem \ref{aiaiaia}, and finally in Section \ref{letmego} we show Theorem
\ref{voices} as well as Lemmas \ref{vortex} and \ref{light}.
\section{Linear problem without viscosity}\label{lv}
In this section, we study the problem
\begin{equation}\label{poissonrocher}
  \begin{cases}
    \partial_t \psi = (f (\bar{\rho}) - 2 \bar{\rho}) \partial_x \psi + \bar{\rho} f^2 (\bar{\rho}) \Delta \varphi  & \quad  (x,y)\in\mathbb{R}^2,\,\,\,\, t\in (0,T)\\
    \partial_t \varphi = f (\bar{\rho}) \partial_x \varphi - \frac{1}{f(\bar{\rho})} \psi &\quad (x,y)\in\mathbb{R}^2, \,\,\,\, t\in (0,T)\\
    \psi (x, y, 0) = \psi_0 (x, y)  &\qquad  (x,y)\in\mathbb{R}^2 \\
    \varphi (x,y, T) = \varphi_T (x, y)&\qquad  (x,y)\in\mathbb{R}^2.
  \end{cases}
\end{equation}
Taking the Fourier transform of this
equation, we write it as
\[ \partial_t \left(\begin{array}{c}
     \hat{\psi}\\
     \hat{\varphi}
   \end{array}\right) = A (\xi) \left(\begin{array}{c}
     \hat{\psi}\\
     \hat{\varphi}
   \end{array}\right) \]
with
\[ A (\xi) = \left(\begin{array}{cc}
     i (f (\bar{\rho}) - 2 \bar{\rho}) \xi_1 & - \bar{\rho} f^2 (\bar{\rho}) |
     \xi |^2\\
     \frac{- 1}{f (\bar{\rho})} & i f (\bar{\rho}) \xi_1
   \end{array}\right), \]
where $ \theta (\xi)$ is given in \eqref{thetaEq}.
\subsection{Diagonalization of $A (\xi)$}
From classical computations we check the following lemma.
\begin{lemma}
  \label{etiq}Take $\xi \in \mathbb{R}^2$ such that $\theta (\xi) \neq 0$.
  Then, the matrix $A (\xi)$ is diagonalizable. Its eigenvalues are \begin{eqnarray*}
    \lambda_1 (\xi) & = & \frac{1}{2} (- \theta (\xi) + 2 i (f (\bar{\rho}) -
    \bar{\rho}) \xi_1)\\
    \lambda_2 (\xi) & = & \frac{1}{2} (\theta (\xi) + 2 i (f (\bar{\rho}) -
    \bar{\rho}) \xi_1),
  \end{eqnarray*}
  with the associated eigenvectors
  \[ V_1 (\xi) \assign \left(\begin{array}{c}
       \frac{f (\bar{\rho})}{2} (\theta (\xi) + 2 i \bar{\rho} \xi_1)\\
       1
     \end{array}\right), V_2 (\xi) \assign \left(\begin{array}{c}
       \frac{f (\bar{\rho})}{2} (- \theta (\xi) + 2 i \bar{\rho} \xi_1)\\
       1
     \end{array}\right) \]
  and we have
  \[ A (\xi) = P \left(\begin{array}{cc}
       \lambda_1 (\xi) & 0\\
       0 & \lambda_2 (\xi)
     \end{array}\right) P^{- 1} \]
  where
  \[ P  = \left(\begin{array}{cc}
       \frac{f (\bar{\rho})}{2} (\theta (\xi) + 2 i \bar{\rho} \xi_1) &
       \frac{f (\bar{\rho})}{2} (- \theta (\xi) + 2 i \bar{\rho} \xi_1)\\
       1 & 1
     \end{array}\right) \]
  and
  \[ P^{- 1} = \left(\begin{array}{cc}
       \frac{1}{f (\bar{\rho}) \theta (\xi)} & \frac{\theta (\xi) - 2 i
       \bar{\rho} \xi_1}{\theta (\xi)}\\
       \frac{- 1}{f (\bar{\rho}) \theta (\xi)} & \frac{\theta (\xi) + 2 i
       \bar{\rho} \xi_1}{\theta (\xi)}
     \end{array}\right) . \]
\end{lemma}
Now, if $\theta (\xi) \neq 0$ we define the vector
\[ \left(\begin{array}{c}
     \hat{u} (\xi, t)\\
     \hat{v} (\xi, t)
   \end{array}\right) \assign P^{- 1} \left(\begin{array}{c}
     \hat{\psi} (\xi, t)\\
     \hat{\varphi} (\xi, t)
   \end{array}\right) \]
and then
\[ \partial_t \left(\begin{array}{c}
     \hat{u}\\
     \hat{v}
   \end{array}\right) = \left(\begin{array}{cc}
     \lambda_1 (\xi) & 0\\
     0 & \lambda_2 (\xi)
   \end{array}\right) \left(\begin{array}{c}
     \hat{u}\\
     \hat{v}
   \end{array}\right) . \]
This decouples the two equations. We have $\partial_t \hat{u} = \lambda_1
(\xi) \hat{u}$ that we integrate forward in time and $\partial_t \hat{v} =
\lambda_2 (\xi) \hat{v}$ that we integrate backward in time. That is,
\begin{equation}
  \hat{u} (\xi, t) = e^{\lambda_1 (\xi) t} \hat{u}_0 (\xi), \hat{v} (\xi, t) =
  e^{- \lambda_2 (\xi) (T - t)} \hat{v}_T (\xi) . \label{focus}
\end{equation}
\subsection{Computation of $\hat{u}_0$ and $\hat{v}_T$}
We recall that $\theta (\xi) = 2 \sqrt{\bar{\rho} (f (\bar{\rho}) -
\bar{\rho}) | \xi_1 |^2 + \bar{\rho} f (\bar{\rho}) | \xi_2 |^2}$. Therefore,
if $f (\bar{\rho}) - \bar{\rho} > 0$ (that is $\bar{\rho} <
\frac{\rho_m}{2}$), there exists a constant $c > 0$ such that
\begin{equation}
  \frac{| \xi |}{c} \geqslant \theta (\xi) \geqslant c | \xi | . \label{rnr}
\end{equation}
\begin{lemma}
  \label{london}Suppose that $\bar{\rho} < \frac{\rho_m}{2}$. Given
  $\hat{\psi}_0 (\xi), \hat{\varphi}_T (\xi)$, define the functions for $\xi
  \neq 0$
  \[ \hat{u}_0 (\xi) = \frac{(\theta (\xi)^2 + 4 \bar{\rho} \xi_1^2) \left(
     \frac{2 e^{- \lambda_2 (\xi) T}}{\theta (\xi) + 2 i \bar{\rho} \xi_1}
     \hat{\varphi}_T (\xi) + \frac{2}{f (\bar{\rho}) (\theta^2 (\xi) + 4
     \bar{\rho}^2 \xi_1^2)} \hat{\psi}_0 (\xi) \right)}{\theta (\xi) (1 + e^{-
     \theta (\xi) T}) + 2 i \bar{\rho} \xi_1 (1 - e^{- \theta (\xi) T})} \]
  and
  \[ \hat{v}_T (\xi) = \frac{(\theta (\xi)^2 + 4 \bar{\rho} \xi_1^2) \left(
     \frac{2 \hat{\varphi}_T (\xi)}{\theta (\xi) - 2 i \bar{\rho} \xi_1} -
     \frac{2 e^{\lambda_1 (\xi) T}}{f (\bar{\rho}) (\theta^2 (\xi) + 4
     \bar{\rho}^2 \xi_1^2)} \hat{\psi}_0 (\xi) \right)}{\theta (\xi) (1 + e^{-
     \theta (\xi) T}) + 2 i \bar{\rho} \xi_1 (1 - e^{- \theta (\xi) T})} . \]
  Then, the functions
  \[ \left(\begin{array}{c}
       \hat{\psi} (\xi, t)\\
       \hat{\varphi} (\xi, t)
     \end{array}\right) = P \left(\begin{array}{c}
       e^{\lambda_1 (\xi) t} \hat{u}_0 (\xi)\\
       e^{- \lambda_2 (\xi) (T - t)} \hat{v}_T (\xi)
     \end{array}\right), \]
  where $P, \lambda_1, \lambda_2$ are defined in Lemma \ref{etiq}, are
  solutions of equation (\ref{poissonrocher}) on $[0, T]$ for $\xi \neq 0$
  with $\hat{\psi} (\xi, 0) = \hat{\psi}_0 (\xi), \hat{\varphi} (\xi, T) =
  \hat{\varphi}_T (\xi)$.
\end{lemma}
\begin{proof}
  From the definition of $\hat{u}, \hat{v}$ and (\ref{focus}), we compute that
  \[ \left(\begin{array}{c}
       \hat{u}_0 (\xi)\\
       e^{- \lambda_2 (\xi) T} \hat{v}_T (\xi)
     \end{array}\right) = P^{- 1} \left(\begin{array}{c}
       \hat{\psi}_0 (\xi)\\
       \hat{\varphi}_0 (\xi)
     \end{array}\right) \]
  and
  \[ \left(\begin{array}{c}
       e^{\lambda_1 (\xi) t} \hat{u}_0 (\xi)\\
       \hat{v}_T (\xi)
     \end{array}\right) = P^{- 1} \left(\begin{array}{c}
       \hat{\psi}_T (\xi)\\
       \hat{\varphi}_T (\xi)
     \end{array}\right) \]
  where $\hat{\varphi}_0 (\xi) = \hat{\varphi} (\xi, 0)$ and $\hat{\psi}_T
  (\xi) = \hat{\psi} (\xi, T)$ are not given. We want to remove them from the
  system.
  \
  The first system can be written by Lemma \ref{etiq} as
  \[ \left\{\begin{array}{l}
       \hat{u}_0 (\xi) = \frac{1}{f (\bar{\rho}) \theta (\xi)} \hat{\psi}_0
       (\xi) + \frac{\theta (\xi) - 2 i \bar{\rho} \xi_1}{\theta (\xi)}
       \hat{\varphi}_0 (\xi)\\\\
       e^{- \lambda_2 (\xi) T} \hat{v}_T (\xi) = \frac{- 1}{f (\bar{\rho})
       \theta (\xi)} \hat{\psi}_0 (\xi) + \frac{\theta (\xi) + 2 i \bar{\rho}
       \xi_1}{\theta (\xi)} \hat{\varphi}_0 (\xi)
     \end{array}\right. \]
  and therefore
  \begin{eqnarray*}
    &  & \frac{1}{\theta (\xi) - 2 i \bar{\rho} \xi_1} \hat{u}_0 (\xi) -
    \frac{1}{\theta (\xi) + 2 i \bar{\rho} \xi_1} e^{- \lambda_2 (\xi) T}
    \hat{v}_T (\xi)\\
    & = & \frac{2}{f (\bar{\rho}) (\theta^2 (\xi) + 4 \bar{\rho}^2 \xi_1^2)}
    \hat{\psi}_0 (\xi) .
  \end{eqnarray*}
  The second system can be written by Lemma \ref{etiq} as
  \[ \left\{\begin{array}{l}
       e^{\lambda_1 (\xi) T} \hat{u}_0 (\xi) = \frac{1}{f (\bar{\rho}) \theta
       (\xi)} \hat{\psi}_T (\xi) + \frac{\theta (\xi) - 2 i \bar{\rho}
       \xi_1}{\theta (\xi)} \hat{\varphi}_T (\xi)\\\\
       \hat{v}_T (\xi) = \frac{- 1}{f (\bar{\rho}) \theta (\xi)} \hat{\psi}_T
       (\xi) + \frac{\theta (\xi) + 2 i \bar{\rho} \xi_1}{\theta (\xi)}
       \hat{\varphi}_T (\xi)
     \end{array}\right. \]
  and therefore summing the two equations yield
  \[ e^{\lambda_1 (\xi) T} \hat{u}_0 (\xi) + \hat{v}_T (\xi) = 2
     \hat{\varphi}_T (\xi) . \]
  We deduce that
  \[ \left(\begin{array}{cc}
       \frac{1}{\theta (\xi) - 2 i \bar{\rho} \xi_1} & \frac{- e^{- \lambda_2
       (\xi) T}}{\theta (\xi) + 2 i \bar{\rho} \xi_1}\\
       e^{\lambda_1 (\xi) T} & 1
     \end{array}\right) \left(\begin{array}{c}
       \hat{u}_0 (\xi)\\
       \hat{v}_T (\xi)
     \end{array}\right) = \left(\begin{array}{c}
       \frac{2}{f (\bar{\rho}) (\theta^2 (\xi) + 4 \bar{\rho}^2 \xi_1^2)}
       \hat{\psi}_0 (\xi)\\
       2 \hat{\varphi}_T (\xi)
     \end{array}\right) . \]
  We have
  \[ D (\xi) = \tmop{Det} \left(\begin{array}{cc}
       \frac{1}{\theta (\xi) - 2 i \bar{\rho} \xi_1} & \frac{- e^{- \lambda_2
       (\xi) T}}{\theta (\xi) + 2 i \bar{\rho} \xi_1}\\
       e^{\lambda_1 (\xi) T} & 1
     \end{array}\right) = \frac{\theta (\xi) (1 + e^{- \theta (\xi) T}) + 2 i
     \bar{\rho} \xi_1 (1 - e^{- \theta (\xi) T})}{\theta (\xi)^2 + 4
     \bar{\rho} \xi_1^2} \neq 0, \]
  for $\xi \neq 0$ since $\theta (\xi)$ is real. Therefore,
  \begin{eqnarray*}
    &  & \left(\begin{array}{c}
      \hat{u}_0 (\xi)\\
      \hat{v}_T (\xi)
    \end{array}\right)\\
    & = & \frac{1}{D (\xi)} \left(\begin{array}{cc}
      1 & \frac{e^{- \lambda_2 (\xi) T}}{\theta (\xi) + 2 i \bar{\rho}
      \xi_1}\\
      - e^{\lambda_1 (\xi) T} & \frac{1}{\theta (\xi) - 2 i \bar{\rho} \xi_1}
    \end{array}\right) \left(\begin{array}{c}
      \frac{2}{f (\bar{\rho}) (\theta^2 (\xi) + 4 \bar{\rho}^2 \xi_1^2)}
      \hat{\psi}_0 (\xi)\\
      2 \hat{\varphi}_T (\xi)
    \end{array}\right)\\
    & = & \frac{1}{D (\xi)} \left(\begin{array}{c}
      \frac{2 e^{- \lambda_2 (\xi) T}}{\theta (\xi) + 2 i \bar{\rho} \xi_1}
      \hat{\varphi}_T (\xi) + \frac{2}{f (\bar{\rho}) (\theta^2 (\xi) + 4
      \bar{\rho}^2 \xi_1^2)} \hat{\psi}_0 (\xi)\\
      \frac{2 \hat{\varphi}_T (\xi)}{\theta (\xi) - 2 i \bar{\rho} \xi_1} -
      \frac{2 e^{\lambda_1 (\xi) T}}{f (\bar{\rho}) (\theta^2 (\xi) + 4
      \bar{\rho}^2 \xi_1^2)} \hat{\psi}_0 (\xi)
    \end{array}\right) .
  \end{eqnarray*}
  This concludes the proof of the lemma.
\end{proof}
\subsection{End of the proof of Proposition \ref{nina}}
We show the first lemma.
\begin{lemma}
  \label{rezident}Suppose that $\bar{\rho} < \frac{\rho_m}{2}$. Then, there exists $K, c > 0$ such that for any $T > 0$, system (\ref{losttimes}) admits a solution on $[0, T]$ that satisfies
  \[
   (| \hat{\psi} | + | \xi \hat{\varphi} |) (\xi, t) \leqslant K (e^{- c |\xi | t} | \hat{\psi}_0 | + e^{- c | \xi | (T - t)} | \xi \hat{\varphi}_T|) .
   \]
\end{lemma}
\begin{proof}
  We recall that $\lambda_1 (\xi) = \frac{1}{2} (- \theta (\xi) + 2 i (f
  (\bar{\rho}) - \bar{\rho}) \xi_1)$ and $\lambda_2 (\xi) = \frac{1}{2}
  (\theta (\xi) + 2 i (f (\bar{\rho}) - \bar{\rho}) \xi_1)$, therefore
  \[ \frac{- | \xi |}{c} \leqslant \mathfrak{R}\mathfrak{e} (\lambda_1 (\xi))
     \leqslant - c | \xi | \]
  and
  \[ \frac{- | \xi |}{c} \leqslant \mathfrak{R}\mathfrak{e} (- \lambda_2
     (\xi)) \leqslant - c | \xi | . \]
  We recall that
  \[ \left(\begin{array}{c}
       \hat{\psi} (\xi, t)\\
       \hat{\varphi} (\xi, t)
     \end{array}\right) = P \left(\begin{array}{c}
       e^{\lambda_1 (\xi) t} \hat{u}_0 (\xi)\\
       e^{- \lambda_2 (\xi) (T - t)} \hat{v}_T (\xi)
     \end{array}\right) \]
  and
  \[ P = \left(\begin{array}{cc}
       \frac{f (\bar{\rho})}{2} (\theta (\xi) + 2 i \bar{\rho} \xi_1) &
       \frac{f (\bar{\rho})}{2} (- \theta (\xi) + 2 i \bar{\rho} \xi_1)\\
       1 & 1
     \end{array}\right), \]
  therefore
  \[ | \hat{\psi} | (\xi, t) \leqslant K (e^{- c | \xi | t} | \xi \hat{u}_0
     (\xi) | + e^{- c | \xi | (T - t)} | \xi \hat{v}_T (\xi) |) \]
  and with Lemma \ref{london} we check that
  \[ | \xi \hat{u}_0 (\xi) | \leqslant K (e^{- c | \xi | T} | \xi
     \hat{\varphi}_T (\xi) | + | \hat{\psi}_0 (\xi) |), \]
  \[ | \xi \hat{v}_T (\xi) | \leqslant K (| \xi \hat{\varphi}_T (\xi) | + e^{-
     c | \xi | T} | \hat{\psi}_0 (\xi) |) . \]
  Since $t \leqslant T$ we deduce that
  \[ | \hat{\psi} | (\xi, t) \leqslant K (e^{- c | \xi | t} | \hat{\psi}_0
     (\xi) | + e^{- c | \xi | (T - t)} | \xi \hat{\varphi}_T (\xi) |) . \]
  We now check that
  \[ | \widehat{\nabla \varphi} | (\xi, t) = | \xi \hat{\varphi} (\xi, t) |
     \leqslant K (e^{- c | \xi | t} | \xi \hat{u}_0 (\xi) | + e^{- c | \xi |
     (T - t)} | \xi \hat{v}_T (\xi) |) \]
  and we can conclude similarly.
\end{proof}

Proposition \ref{nina} follows then directly from classical computations.
\section{Proof of Theorem \ref{aiaiaia}}\label{alive}
\subsection{The Duhamel formulation}\label{embrz}
We write the system (\ref{mq1})-(\ref{mq2}) with $\sigma = 0$ under the form
\[ \left\{\begin{array}{l}
     \partial_t \psi = (f (\bar{\rho}) - 2 \bar{\rho}) \partial_x \psi +
     \bar{\rho} f^2 (\bar{\rho}) \Delta \varphi + \tmop{NL}_1 (\psi, \nabla
     \varphi)\\
     \partial_t \varphi = f (\bar{\rho}) \partial_x \varphi - \frac{1}{f
     (\bar{\rho})} \psi + \tmop{NL}_2 (\psi, \nabla \varphi)
   \end{array}\right. \]
where
\begin{eqnarray*}
  \tmop{NL}_1 (\psi, \varphi) & \assign & \left( \frac{\bar{\rho}}{f
  (\bar{\rho})} - 2 \right) \partial_x (\psi^2) + \frac{1}{f (\bar{\rho})}
  \partial_x (\psi^3)\\
  & + & \nabla . ((\psi f^2 (\bar{\rho}) - 2 f (\bar{\rho}) (\bar{\rho} +
  \psi) \psi + (\bar{\rho} + \rho) \psi^2) \nabla \varphi)
\end{eqnarray*}
and
\begin{eqnarray*}
  \tmop{NL}_2 (\psi, \varphi) & \assign & \frac{1}{2} f^2 (\bar{\rho}) |
  \nabla \varphi |^2 + \frac{1}{2} (\psi^2 - 2 f (\bar{\rho}) \psi) | \nabla
  \varphi |^2\\
  & + & \frac{1}{2 f (\bar{\rho})} \psi^2 + \partial_x \varphi \left(
  \frac{1}{f (\bar{\rho})} \psi^2 - 2 \psi \right) .
\end{eqnarray*}
Taking the Fourier transform of this system, we infer that
\[ \partial_t \left(\begin{array}{c}
     \hat{\psi}\\
     \hat{\varphi}
   \end{array}\right) = A (\xi) \left(\begin{array}{c}
     \hat{\psi}\\
     \hat{\varphi}
   \end{array}\right) + \left(\begin{array}{c}
     \widehat{\tmop{NL}_1} (\psi, \varphi)\\
     \widehat{\tmop{NL}_2} (\psi, \varphi)
   \end{array}\right), \]
where the matrix $A$ is defined in (\ref{matA}). Therefore, defining
\begin{equation}
  \left(\begin{array}{c}
    \hat{u}\\
    \hat{v}
  \end{array}\right) = P^{- 1} \left(\begin{array}{c}
    \hat{\psi}\\
    \hat{\varphi}
  \end{array}\right)
\end{equation}
($P$ is defined in Lemma \ref{etiq}), we have
\[ \partial_t \left(\begin{array}{c}
     | \xi | \hat{u}\\
     | \xi | \hat{v}
   \end{array}\right) = \left(\begin{array}{cc}
     \lambda_1 (\xi) & 0\\
     0 & \lambda_2 (\xi)
   \end{array}\right) \left(\begin{array}{c}
     | \xi | \hat{u}\\
     | \xi | \hat{v}
   \end{array}\right) + | \xi | P^{- 1} \left(\begin{array}{c}
     \widehat{\tmop{NL}_1} (\psi, \varphi)\\
     \widehat{\tmop{NL}_2} (\psi, \varphi)
   \end{array}\right) . \]
This leads to the following Duhamel formulation.
\begin{lemma}
  We define the functional operators
  \[ S_1 (\hat{\psi}, \hat{\varphi}) (\xi, t) \assign \int_0^t e^{\lambda_1
     (\xi) (t - s)} | \xi | \left( P^{- 1} \left(\begin{array}{c}
       \widehat{\tmop{NL}_1} (\psi, \varphi)\\
       \widehat{\tmop{NL}_2} (\psi, \varphi)
     \end{array}\right) (\xi, s) . \left(\begin{array}{c}
       1\\
       0
     \end{array}\right) \right) d s \]
  and
  \[ S_2 (\hat{\psi}, \hat{\varphi}) (\xi, t) \assign \int_T^{T - t}
     e^{\lambda_2 (\xi) (T - t - s)} | \xi | \left( P^{- 1}
     \left(\begin{array}{c}
       \widehat{\tmop{NL}_1} (\psi, \varphi)\\
       \widehat{\tmop{NL}_2} (\psi, \varphi)
     \end{array}\right) (\xi, s) . \left(\begin{array}{c}
       0\\
       1
     \end{array}\right) \right) d s. \]
  Then, a solution of equations (\ref{mq1})-(\ref{mq2}) can be written in the
  form
  \[ \left(\begin{array}{c}
       \hat{\psi}\\
       \hat{\varphi}
     \end{array}\right) (\xi, t) = P (\xi) \left(\begin{array}{c}
       \hat{u}\\
       \hat{v}
     \end{array}\right) (\xi, t), \]
  \[ \left(\begin{array}{c}
       | \xi | \hat{u}\\
       | \xi | \hat{v}
     \end{array}\right) (\xi, t) = \left(\begin{array}{c}
       e^{\lambda_1 (\xi) t} | \xi | \hat{u}_0 (\xi)\\
       e^{\lambda_1 (\xi) (T - t)} | \xi | \hat{v}_T (\xi)
     \end{array}\right) + \left(\begin{array}{c}
       S_1 (\hat{\psi}, \hat{\varphi}) (\xi, t)\\
       S_2 (\hat{\psi}, \hat{\varphi}) (\xi, t)
     \end{array}\right) . \]
  That is, if there exists ($\hat{\psi}, \hat{\varphi}$) solution of this
  problem, then it also solves \ (\ref{mq1})-(\ref{mq2}).
\end{lemma}
\subsection{Estimates of some convolutions}\label{lane8}
We define the norm
\[ \| \hat{f} (., t) \| \assign \left\| \frac{(1 + | \xi |)^3 \hat{f}}{e^{- c
   | \xi | t} + e^{- c | \xi | (T - t)}} \right\|_{L^{\infty} (\mathbb{R}^2)}
   . \]
\begin{lemma}
  \label{themusic}There exists $K > 0$ such that for any function $f, g$ such
  that $\| \hat{f} \|, \| \hat{g} \| < + \infty$, we have
  \[ \| \widehat{(f g)} \| \leqslant K \| \hat{f} \| \| \hat{g} \| \]
\end{lemma}
\begin{proof}
  We have
  \[ \widehat{(f g)} (\xi) = \int_{\mathbb{R}^2} \hat{f} (\nu) \hat{g} (\xi -
     \nu) d \nu \]
  and thus
  \begin{eqnarray*}
    &  & \frac{| \widehat{(f g)} (\xi) |}{\| \hat{f} \| \| \hat{g} \|}\\
    & \leqslant & \int_{\mathbb{R}^2} \frac{(e^{- c | \nu | t} + e^{- c | \nu
    | (T - t)}) (e^{- c | \xi - \nu | t} + e^{- c | \xi - \nu | (T - t)})}{(1
    + | \nu |)^3 (1 + | \xi - \nu |)^3} d \nu
  \end{eqnarray*}
  and we infer that
  \begin{eqnarray*}
    &  & (e^{- c | \nu | t} + e^{- c | \nu | (T - t)}) (e^{- c | \xi - \nu |
    t} + e^{- c | \xi - \nu | (T - t)})\\
    & \leqslant & K (e^{- c | \xi | t} + e^{- c | \xi | (T - t)})
  \end{eqnarray*}
  for any $\nu, \xi \in \mathbb{R}^2, 0 \leqslant t \leqslant T$. Indeed, by
  triangular inequality we have
  \[ e^{- c | \nu | t} e^{- c | \xi - \nu | t} \leqslant e^{- c | \xi | t}, \]
  and separating the cases $t \leqslant \frac{T}{2}$ and $t \geqslant
  \frac{T}{2}$ we check that
  \[ e^{- c | \nu | t} e^{- c | \xi - \nu | (T - t)} \leqslant e^{- c | \xi |
     t} + e^{- c | \xi | (T - t)} . \]
  We check that all other products can be estimates similarly.
  Furthermore, since $\frac{1}{(1 + | \xi |)^3} \in L^1 (\mathbb{R}^2)$ we
  check by standard estimates that
  \[ \int_{\mathbb{R}^2} \frac{1}{(1 + | \nu |)^3 (1 + | \xi - \nu |)^3} d \nu
     \leqslant \frac{K}{(1 + | \xi |)^3}, \]
  hence
  \[ | \widehat{(f g)} (\xi) | \leqslant K \| \hat{f} \| \| \hat{g} \|
     \frac{(e^{- c | \xi | t} + e^{- c | \xi | (T - t)})}{(1 + | \xi |)^3}, \]
  concluding the proof of the lemma.
\end{proof}
\begin{lemma}
  \label{gettingstronger}There exists $K > 0$ such that for any function $f$
  such that $\| \hat{f} \| < + \infty$, we have
  \[ \left\| \int_0^t e^{- 2 c | \xi | (t - s)} | \xi | \hat{f} (\xi, s) d s
     \right\| \leqslant K \| \hat{f} \| . \]
\end{lemma}
\begin{proof}
  We check that
  \begin{eqnarray*}
    &  & \left| \int_0^t e^{- 2 c | \xi | (t - s)} | \xi | \hat{f} (\xi, s) d
    s \right| (1 + | \xi |)^3\\
    & \leqslant & \| \hat{f} \| \int_0^t e^{- 2 c | \xi | (t - s)} | \xi |
    (e^{- c | \xi | s} + e^{- c | \xi | (T - s)}) d s.
  \end{eqnarray*}
  Now, we compute that
  \[ \int_0^t e^{- 2 c | \xi | (t - s)} | \xi | e^{- c | \xi | s} d s
     \leqslant e^{- c | \xi | t} \int_0^t e^{- c | \xi | (t - s)} | \xi | d s
     \leqslant K e^{- c | \xi | t} \]
  and
  \[ \int_0^t e^{- 2 c | \xi | (t - s)} | \xi | e^{- c | \xi | (T - s)} d s
     \leqslant e^{- c | \xi | (T - t)} \int_0^t e^{- c | \xi | (t - s)} | \xi
     | d s \leqslant e^{- c | \xi | (T - t)}, \]
  which concludes the proof of this lemma.
\end{proof}
\subsection{Conclusion of the fixed point argument}
\subsubsection{Estimates on $\| \hat{\psi} \| + \| \widehat{\nabla \varphi}
\|$ by $\| \xi \hat{u} \| + \| \xi \hat{v} \|$}
We recall that the norm $\| . \|$ is defined in subsection \ref{lane8}.
\begin{lemma}
  Suppose that $\bar{\rho} < \frac{\rho_m}{2}$. Then, there exists a constant
  $K > 0$ such that
  \[ \| \hat{\psi} \| + \| \widehat{\nabla \varphi} \| \leqslant K (\| \xi
     \hat{u} \| + \| \xi \hat{v} \|) \]
  if
  \[ \left(\begin{array}{c}
       \hat{\psi}\\
       \hat{\varphi}
     \end{array}\right) (\xi, t) = P \left(\begin{array}{c}
       \hat{u}\\
       \hat{v}
     \end{array}\right) (\xi, t) \]
  where $P$ is defined in Lemma \ref{etiq}.
\end{lemma}
\begin{proof}
  We recall that
  \[ P = \left(\begin{array}{cc}
       \frac{f (\bar{\rho})}{2} (\theta (\xi) + 2 i \bar{\rho} \xi_1) &
       \frac{f (\bar{\rho})}{2} (- \theta (\xi) + 2 i \bar{\rho} \xi_1)\\
       1 & 1
     \end{array}\right) \]
  and therefore, with (\ref{rnr}) we have
  \[ | \hat{\psi} | \leqslant K (| \xi \hat{u} | + | \xi \hat{v} |) \]
  and
  \[ | \hat{\varphi} | \leqslant K (| \hat{u} | + | \hat{v} |), \]
  concluding the proof of the lemma.
\end{proof}
\subsubsection{Estimates on the nonlinear parts}
\begin{lemma}
  Suppose that $\bar{\rho} < \frac{\rho_m}{2}$. Then, there exists $c, K > 0$
  such that, for any $T > 0$ and any $\psi, \nabla \varphi$ such that $\|
  \hat{\psi} \|, \| \widehat{\nabla \varphi} \| < + \infty$, we have
  \[ \| S_1 (\hat{\psi}, \hat{\varphi}) \| + \| S_2 (\hat{\psi},
     \hat{\varphi}) \| \leqslant K (\| \hat{\psi} \| + \| \widehat{\nabla
     \varphi} \|)^2 (1 + \| \hat{\psi} \| + \| \widehat{\nabla \varphi} \|)^2
     . \]
  As such, the function $S_1, S_2$ are contractions on the space of functions
  $\psi, \nabla \varphi$ such that $\| \hat{\psi} \| + \| \widehat{\nabla
  \varphi} \| \leqslant \varepsilon_0$ for some small universal constant
  $\varepsilon_0$.
\end{lemma}
\begin{proof}
  We recall that
  \[ S_1 (\hat{\psi}, \hat{\varphi}) (\xi, t) = \int_0^t e^{\lambda_1 (\xi) (t
     - s)} | \xi | \left( P^{- 1} \left(\begin{array}{c}
       \widehat{\tmop{NL}_1} (\psi, \varphi)\\
       \widehat{\tmop{NL}_2} (\psi, \varphi)
     \end{array}\right) (\xi, s) . \left(\begin{array}{c}
       1\\
       0
     \end{array}\right) \right) d s. \]
  By Lemma \ref{gettingstronger} we have, for $c > 0$ small enough that
  \[ \| S_1 (\hat{\psi}, \hat{\varphi}) \| \leqslant K \left\| P^{- 1}
     \left(\begin{array}{c}
       \widehat{\tmop{NL}_1} (\psi, \varphi)\\
       \widehat{\tmop{NL}_2} (\psi, \varphi)
     \end{array}\right) . \left(\begin{array}{c}
       1\\
       0
     \end{array}\right) \right\| . \]
  We recall that
  \[ P^{- 1} = \left(\begin{array}{cc}
       \frac{1}{f (\bar{\rho}) \theta (\xi)} & \frac{\theta (\xi) - 2 i
       \bar{\rho} \xi_1}{\theta (\xi)}\\
       \frac{- 1}{f (\bar{\rho}) \theta (\xi)} & \frac{\theta (\xi) + 2 i
       \bar{\rho} \xi_1}{\theta (\xi)}
     \end{array}\right) \]
  and
  \begin{eqnarray*}
    \tmop{NL}_1 (\psi, \varphi) & = & \left( \frac{\bar{\rho}}{f (\bar{\rho})}
    - 2 \right) \partial_x (\psi^2) + \frac{1}{f (\bar{\rho})} \partial_x
    (\psi^3)\\
    & + & \nabla . ((\psi f^2 (\bar{\rho}) - 2 f (\bar{\rho}) (\bar{\rho} +
    \psi) \psi + (\bar{\rho} + \psi) \psi^2) \nabla \varphi) .
  \end{eqnarray*}
  A key remark here is that all the terms $\tmop{NL}_1 (\psi, \varphi)$ are
  the derivative of some quantities, and thus we check that
  \[ \frac{| \widehat{\tmop{NL}_1} (\psi, \varphi) | (\xi, t)}{| \xi |}
     \leqslant K (| \widehat{(\psi^2)} | + | \widehat{(\psi^3)} | + |
     \widehat{(\psi^2 \nabla \varphi)} | + | \widehat{(\psi^3 \nabla \varphi)}
     |) . \]
  We recall that
  \begin{eqnarray*}
    \tmop{NL}_2 (\psi, \varphi) & = & \frac{1}{2} f^2 (\bar{\rho}) | \nabla
    \varphi |^2 + \frac{1}{2} (\psi^2 - 2 f (\bar{\rho}) \psi) | \nabla
    \varphi |^2\\
    & + & \frac{1}{2 f (\bar{\rho})} \psi^2 + \partial_x \varphi \left(
    \frac{1}{f (\bar{\rho})} \psi^2 - 2 \psi \right)
  \end{eqnarray*}
  and thus
  \begin{eqnarray*}
    | \widehat{\tmop{NL}_2} (\psi, \varphi) | (\xi, t) & \leqslant & K (|
    \widehat{(| \nabla \varphi |^2)} | + | \widehat{(\psi | \nabla \varphi
    |^2)} | + | \widehat{(\psi^2 | \nabla \varphi |^2)} |)\\
    & + & K (| \widehat{(\psi^2)} | + | \widehat{(\partial_x \varphi \psi)} |
    + | \widehat{(\partial_x \varphi \psi^2)} |) .
  \end{eqnarray*}
  Now, we check easily that if $\bar{\rho} < \frac{\rho_m}{2}$, then
  \[ \left| P^{- 1} \left(\begin{array}{c}
       \widehat{\tmop{NL}_1} (\psi, \varphi)\\
       \widehat{\tmop{NL}_2} (\psi, \varphi)
     \end{array}\right) \right| \leqslant K \left( \frac{|
     \widehat{\tmop{NL}_1} (\psi, \varphi) |}{| \xi |} + |
     \widehat{\tmop{NL}_2} (\psi, \varphi) | \right) \]
  and we can then conclude using Lemma \ref{themusic}. A similar proof holds
  for $S_2 (\hat{\psi}, \hat{\varphi})$.
\end{proof}
\subsubsection{End of the proof of Theorem \ref{aiaiaia}}
\begin{proof}
  Take an intial data that satisfies
  \[ \| (1 + | \xi |)^3 \hat{\psi}_0 \|_{L^{\infty} (\mathbb{R}^2)} + \| (1 +
     | \xi |)^3 \widehat{\nabla \varphi_T} \|_{L^{\infty} (\mathbb{R}^2)}
     \leqslant \varepsilon \]
  for some small $\varepsilon_0$ determined later on. Then, we can construct
  by Proposition \ref{nina} and Lemma \ref{rezident} a solution $\Psi, \Phi$
  to the linear problem that satisfy
  \[ \| \hat{\Psi} \| + \| \widehat{\nabla \Phi} \| \leqslant K \varepsilon .
  \]
  By the fixed point Theorem, we deduce that taking $\varepsilon$ small
  enough, we can construct a solution of the full nonlinear problem $\psi,
  \varphi$, satisfying
  \[ \| \hat{\psi} \| + \| \widehat{\nabla \varphi} \| \leqslant K \varepsilon
     . \]
  Finally, Theorem \ref{aiaiaia} is more precise than this since we want to
  show that
  \[ (\| \psi \|_{L^2} + \| \nabla \varphi \|_{L^2}) (t) \leqslant K \left(
     \frac{\| (1 + | \xi |)^3 \hat{\psi}_0 \|_{L^{\infty} (\mathbb{R}^2)}}{(1
     + t)} + \frac{\| (1 + | \xi |)^3 \widehat{\nabla \varphi_T}
     \|_{L^{\infty} (\mathbb{R}^2)}}{(1 + T - t)} \right) \]
  for instance and not simply
  \[ (\| \psi \|_{L^2} + \| \nabla \varphi \|_{L^2}) (t) \leqslant K
     \varepsilon \left( \frac{1}{(1 + t)} + \frac{1}{(1 + T - t)} \right) \]
  which is a direct consequence of $\| \hat{\psi} \| + \| \widehat{\nabla
  \varphi} \| \leqslant K \varepsilon .$ To do so, we simply check that in
  Lemma \ref{themusic} we have that for $\varepsilon_1, \varepsilon_2 > 0$,
  \begin{eqnarray*}
    &  & (\varepsilon_1 e^{- c | \nu | t} + \varepsilon_2 e^{- c | \nu | (T -
    t)}) (\varepsilon_1 e^{- c | \xi - \nu | t} + \varepsilon_2 e^{- c | \xi -
    \nu | (T - t)})\\
    & \leqslant & K \min (\varepsilon_1, \varepsilon_2) (\varepsilon_1 e^{- c
    | \xi | t} + \varepsilon_2 e^{- c | \xi | (T - t)})
  \end{eqnarray*}
  and in Lemma \ref{gettingstronger} the contributions of $e^{- c | \nu | t}$
  and $e^{- c | \nu | (T - t)}$ are treated separately.
\end{proof}
\section{Proof of Theorem \ref{voices} and Lemma \ref{vortex}}\label{letmego}
\subsection{Existence and estimates of solutions for $\sigma > 0$}
Here, we prove the following first part of Theorem \ref{voices}.
\begin{lemma}
  Suppose that $\bar{\rho} < \frac{\rho_m}{2}$. Then, for any $\sigma > 0$, there exists $K, \varepsilon_0 > 0$ such that equation
  (\ref{mq1})-(\ref{mq2}) for an initial data $\psi (x, y, 0) = \psi_0 (x, y)$
  and $\varphi (x, y, T) = \varphi_T (x, y)$ \ with
  \[ \| (1 + | \xi |)^3 \hat{\psi}_0 \|_{L^{\infty} (\mathbb{R}^2)} + \| (1 +
     | \xi |)^3 (| \xi | + \sigma | \xi |^2) \widehat{\varphi_T}
     \|_{L^{\infty} (\mathbb{R}^2)} \leqslant \varepsilon_0 \]
  admits a solution ($\psi_{\sigma}, \varphi_{\sigma}$) in $L^2 \cap
  L^{\infty}$ on $[0, T]$ and this solution satisfies for all $t \in [0, T]$
  that
  \begin{eqnarray*}
    &  & (\| \psi_{\sigma} \|_{L^2} + \| \nabla \varphi_{\sigma} \|_{L^2} +
    \sigma \| \nabla^2 \varphi_{\sigma} \|_{L^2}) (t)\\
    & \leqslant & K \left( \frac{\| (1 + | \xi |)^3(1+\sigma|\xi|) \hat{\psi}_0
    \|_{L^{\infty} (\mathbb{R}^2)}}{(1 + t)} + \frac{\| (1 + | \xi |)^3 (| \xi
    | + \sigma | \xi |^2) \widehat{\varphi_T} \|_{L^{\infty}
    (\mathbb{R}^2)}}{(1 + T - t)} \right)
  \end{eqnarray*}
  and
  \begin{eqnarray*}
    &  & (\| \psi_{\sigma} \|_{L^{\infty}} + \| \nabla \varphi_{\sigma}
    \|_{L^{\infty}} + \sigma \| \nabla^2 \varphi_{\sigma} \|_{L^{\infty}})
    (t)\\
    & \leqslant & K \left( \frac{\| (1 + | \xi |)^3 (1+\sigma|\xi|)\hat{\psi}_0
    \|_{L^{\infty} (\mathbb{R}^2)}}{(1 + t)^2} + \frac{\| (1 + | \xi |)^3 (|
    \xi | + \sigma | \xi |^2) \widehat{\varphi_T} \|_{L^{\infty}
    (\mathbb{R}^2)}}{(1 + T - t)^2} \right) .
  \end{eqnarray*}
\end{lemma}
\begin{proof}
  We write the system (\ref{mq1})-(\ref{mq2}) (now for $\sigma \neq 0$) in
  Fourier in the form
  \[ \partial_t \left(\begin{array}{c}
       \hat{\psi}_{\sigma}\\
       \hat{\varphi}_{\sigma}
     \end{array}\right) = A_{\sigma} (\xi) \left(\begin{array}{c}
       \hat{\psi}_{\sigma}\\
       \hat{\varphi}_{\sigma}
     \end{array}\right) + \left(\begin{array}{c}
       \widehat{\tmop{NL}_1} (\psi_{\sigma}, \varphi_{\sigma})\\
       \widehat{\tmop{NL}_2} (\psi_{\sigma}, \varphi_{\sigma})
     \end{array}\right), \]
  where $\tmop{NL}_1, \tmop{NL}_2$ are the same terms as in subsection
  \ref{embrz}, and
  \[ A_{\sigma} (\xi) \assign \left(\begin{array}{cc}
       i (f (\bar{\rho}) - 2 \bar{\rho}) \xi_1 - \sigma | \xi |^2 & -
       \bar{\rho} f^2 (\bar{\rho}) | \xi |^2\\
       \frac{- 1}{f (\bar{\rho})} & i f (\bar{\rho}) \xi_1 + \sigma | \xi |^2
     \end{array}\right) . \]
  We check by explicit computation that the eigenvalues of $A_{\sigma} (\xi)$
  are
  \[ \lambda_{1, \sigma} (\xi) \assign \frac{1}{2} (- \theta_{\sigma} (\xi) +
     2 i (f (\bar{\rho}) - \bar{\rho}) \xi_1) \]
  and
  \[ \lambda_{2, \sigma} (\xi) \assign \frac{1}{2} (\theta_{\sigma} (\xi) + 2
     i (f (\bar{\rho}) - \bar{\rho}) \xi_1), \]
  where
  \[ \theta_{\sigma} (\xi) \assign 2 \sqrt{f (\bar{\rho}) \bar{\rho} | \xi |^2
     - (\bar{\rho} \xi_1 - i \sigma | \xi |^2)^2} . \]
  Furthermore, we have
  \[ A_{\sigma} (\xi) = P_{\sigma} \left(\begin{array}{cc}
       \lambda_{1, \sigma} (\xi) & 0\\
       0 & \lambda_{2, \sigma} (\xi)
     \end{array}\right) P^{- 1}_{\sigma} \]
  where
  \[ P_{\sigma} \assign \left(\begin{array}{cc}
       f (\bar{\rho}) \left( \frac{\theta_{\sigma} (\xi)}{2} + i \bar{\rho}
       \xi_1 + \sigma | \xi |^2 \right) & f (\bar{\rho}) \left( -
       \frac{\theta_{\sigma} (\xi)}{2} + i \bar{\rho} \xi_1 + \sigma | \xi |^2
       \right)\\
       1 & 1
     \end{array}\right) \]
  and
  \[ P_{\sigma}^{- 1} = \left(\begin{array}{cc}
       \frac{1}{f (\bar{\rho}) \theta_{\sigma} (\xi)} & \frac{\theta_{\sigma}
       (\xi) / 2 - i \bar{\rho} \xi_1 - \sigma | \xi |^2}{\theta_{\sigma}
       (\xi)}\\
       \frac{- 1}{f (\bar{\rho}) \theta_{\sigma} (\xi)} &
       \frac{\theta_{\sigma} (\xi) / 2 + i \bar{\rho} \xi_1 + \sigma | \xi
       |^2}{\theta_{\sigma} (\xi)}
     \end{array}\right) . \]
  Now, we check that
  \[ \theta_{\sigma} (\xi) = 2 \sqrt{f (\bar{\rho}) \bar{\rho} | \xi |^2 -
     \bar{\rho}^2 \xi_1^2 + \sigma^2 | \xi |^4 + 2 i \sigma \bar{\rho} \xi_1 |
     \xi |^2} \]
  and in particular, if $\bar{\rho} < \frac{\rho_m}{2}$, $\sigma > 0$, there
  exists $c > 0$ such that
  \[ \frac{1}{c} (| \xi | + \sigma | \xi |^2) \geqslant
     \mathfrak{R}\mathfrak{e} (\theta_{\sigma} (\xi)) \geqslant c (| \xi | +
     \sigma | \xi |^2) . \]
  We will follow a similar proof as the one of Theorem \ref{aiaiaia}.
  Concerning the linear problem, we define similarly
  \[ \left(\begin{array}{c}
       \hat{u}_{\sigma}\\
       \hat{v}_{\sigma}
     \end{array}\right) = P_{\sigma}^{- 1} \left(\begin{array}{c}
       \hat{\psi}_{\sigma}\\
       \hat{\varphi}_{\sigma}
     \end{array}\right) \]
  and we check, with similar computations as Lemmas \ref{london} that
  \begin{eqnarray*}
    &  & \left(\begin{array}{cc}
      f (\bar{\rho}) (\theta_{\sigma} / 2 + i \bar{\rho} \xi_1 + \sigma | \xi
      |^2) & - e^{- \lambda_{2, \sigma} T} f (\bar{\rho}) (\theta_{\sigma} / 2
      - i \bar{\rho} \xi_1 - \sigma | \xi |^2)\\
      e^{\lambda_{1, \sigma} T} & 1
    \end{array}\right) \left(\begin{array}{c}
      \hat{u}_{\sigma} (\xi, 0)\\
      \hat{v}_{\sigma} (\xi, T)
    \end{array}\right)\\
    & = & \left(\begin{array}{c}
      \hat{\psi}_0 (\xi)\\
      \hat{\varphi}_T (\xi)
    \end{array}\right) .
  \end{eqnarray*}
  We have
  \begin{eqnarray*}
    &  & D_{\sigma} (\xi)\\
    & \assign & \tmop{Det} \left(\begin{array}{cc}
      f (\bar{\rho}) (\theta_{\sigma} / 2 + i \bar{\rho} \xi_1 + \sigma | \xi
      |^2) & - e^{- \lambda_{2, \sigma} T} f (\bar{\rho}) (\theta_{\sigma} / 2
      - i \bar{\rho} \xi_1 - \sigma | \xi |^2)\\
      e^{\lambda_{1, \sigma} T} & 1
    \end{array}\right)\\
    & = & f (\bar{\rho}) \left( \frac{\theta_{\sigma}}{2} (1 + e^{-
    \theta_{\sigma} (\xi) T}) + \sigma | \xi |^2 (1 - e^{- \theta_{\sigma}
    (\xi) T}) + i \bar{\rho} \xi_1 (1 - e^{- \theta_{\sigma} (\xi) T})
    \right),
  \end{eqnarray*}
  and we check that for $\xi \neq 0$, this quantity is not $0$. We deduce that
  \begin{eqnarray*}
    \hat{u}_{\sigma} (\xi, 0) & = & \frac{\hat{\psi}_0 (\xi) + e^{-
    \lambda_{2, \sigma} T} f (\bar{\rho}) (\theta_{\sigma} / 2 - i \bar{\rho}
    \xi_1 - \sigma | \xi |^2) \hat{\varphi}_T (\xi)}{D_{\sigma} (\xi)}\\
    \hat{v}_{\sigma} (\xi, T) & = & \frac{- e^{\lambda_{1, \sigma} T}
    \hat{\psi}_0 (\xi) + f (\bar{\rho}) (\theta_{\sigma} / 2 + i \bar{\rho}
    \xi_1 + \sigma | \xi |^2) \hat{\varphi}_T (\xi)}{D_{\sigma} (\xi)}
  \end{eqnarray*}
  and
  \[ \left(\begin{array}{c}
       \hat{\psi}_{\sigma}\\
       \hat{\varphi}_{\sigma}
     \end{array}\right) (\xi, t) = P_{\sigma} \left(\begin{array}{c}
       e^{\lambda_{1, \sigma} (\xi) t} \hat{u}_{\sigma} (\xi, 0)\\
       e^{- \lambda_{2, \sigma} (\xi) (T - t)} \hat{v}_{\sigma} (\xi, T)
     \end{array}\right) . \]
  First, we check that there exists $\kappa > 0$ small such that
  \[ | D_{\sigma} (\xi) | \geqslant \kappa (| \xi | + \sigma | \xi |^2) . \]
  From this, we deduce the following estimate for the linear problem:
  \begin{eqnarray*}
    &  & (| \hat{\psi}_{\sigma} | + (| \xi | + \sigma | \xi |^2) |
    \hat{\varphi}_{\sigma} |) (\xi, t)\\
    & \leqslant & K (e^{- c (| \xi | + \sigma | \xi |^2) t} | \hat{\psi}_0 |
    + e^{- c (| \xi | + \sigma | \xi |^2) (T - t)} (| \xi | + \sigma | \xi
    |^2) | \hat{\varphi}_T |) .
  \end{eqnarray*}
  Now, we check, with a similar proof, that the following variation of Lemma
  \ref{gettingstronger} holds:
  \[ \left\| \int_0^t e^{- 2 c (| \xi | + \sigma | \xi |^2) (t - s)} (| \xi |
     + \sigma | \xi |^2) \hat{f} (\xi, s) d s \right\| \leqslant K \| \hat{f}
     \| . \]
  We can then conclude as the end of the proof of Theorem \ref{aiaiaia}.
  Remark that we did not change the norm $\| . \|$ here, it does not depend on
  $\sigma$. If we tried to bootstrap the fact that $\hat{f}$ decay like $e^{-
  c (| \xi | + \sigma | \xi |^2) t}$ instead of $e^{- c | \xi | t}$, the
  triangular inequality used in Lemma \ref{themusic} to estimates convolutions
  with the norm $\| . \|$ would not work.
\end{proof}
\subsection{Vanishing viscosity limit for the linear problem}
We conclude here the proof of Theorem \ref{voices}.

\begin{proof}
 Denoting $\psi, \varphi$ the solution of (\ref{mq1})-(\ref{mq2}) for $\sigma = 0$ for a given
initial data $\psi_0, \varphi_T$, and $\psi_{\sigma}, \varphi_{\sigma}$ the
solution for the same equation with $\sigma > 0$ for the same initial data, we check that
they satisfy the equations
\[ \partial_t \left(\begin{array}{c}
     \hat{\psi}\\
     \hat{\varphi}
   \end{array}\right) = A (\xi) \left(\begin{array}{c}
     \hat{\psi}\\
     \hat{\varphi}
   \end{array}\right) + \left(\begin{array}{c}
     \widehat{\tmop{NL}}_1 (\psi, \varphi)\\
     \widehat{\tmop{NL}}_2 (\psi, \varphi)
   \end{array}\right) \]
and
\[ \partial_t \left(\begin{array}{c}
     \hat{\psi}_{\sigma}\\
     \hat{\varphi}_{\sigma}
   \end{array}\right) = A (\xi) \left(\begin{array}{c}
     \hat{\psi}_{\sigma}\\
     \hat{\varphi}_{\sigma}
   \end{array}\right) + \left(\begin{array}{c}
     \widehat{\tmop{NL}}_1 (\psi_{\sigma}, \varphi_{\sigma})\\
     \widehat{\tmop{NL}}_2 (\psi_{\sigma}, \varphi_{\sigma})
   \end{array}\right) + \sigma \left(\begin{array}{c}
     \widehat{\Delta \psi_{\sigma}}\\
     - \widehat{\Delta \varphi_{\sigma}}
   \end{array}\right) . \]
Therefore,
\begin{eqnarray*}
  &  & \partial_t \left(\begin{array}{c}
    \widehat{\psi_{\sigma} - \psi}\\
    \widehat{\varphi_{\sigma} - \varphi}
  \end{array}\right)\\
  & = & A (\xi) \left(\begin{array}{c}
    \widehat{\psi_{\sigma} - \psi}\\
    \widehat{\varphi_{\sigma} - \varphi}
  \end{array}\right) + \left(\begin{array}{c}
    \widehat{\tmop{NL}}_1 (\psi_{\sigma}, \varphi_{\sigma})\\
    \widehat{\tmop{NL}}_2 (\psi_{\sigma}, \varphi_{\sigma})
  \end{array}\right) - \left(\begin{array}{c}
    \widehat{\tmop{NL}}_1 (\psi, \varphi)\\
    \widehat{\tmop{NL}}_2 (\psi, \varphi)
  \end{array}\right)\\
  & + & \sigma \left(\begin{array}{c}
    \widehat{\Delta \psi_{\sigma}}\\
    - \widehat{\Delta \varphi_{\sigma}}
  \end{array}\right),
\end{eqnarray*}
and by the Duhamel formulation we write that with
\begin{equation}
  \left(\begin{array}{c}
    \widehat{\psi_{\sigma} - \psi}\\
    \widehat{\varphi_{\sigma} - \varphi}
  \end{array}\right) = P \left(\begin{array}{c}
    \hat{u}\\
    \hat{v}
  \end{array}\right), \label{solee}
\end{equation}
we have (with the same notation as in subsection \ref{embrz})
\[ \left(\begin{array}{c}
     | \xi | \hat{u}\\
     | \xi | \hat{v}
   \end{array}\right) = \left(\begin{array}{c}
     S_1 (\hat{\psi}_{\sigma}, \hat{\varphi}_{\sigma}) - S_1 (\hat{\psi},
     \hat{\varphi})\\
     S_2 (\hat{\psi}_{\sigma}, \hat{\varphi}_{\sigma}) - S_2 (\hat{\psi},
     \hat{\varphi})
   \end{array}\right) + \sigma \left(\begin{array}{c}
     \tilde{S}_1\\
     \tilde{S}_2
   \end{array}\right) \]
where
\[ \tilde{S}_1 = \int_0^t e^{\lambda_1 (\xi) (t - s)} | \xi | \left( P^{- 1}
   \left(\begin{array}{c}
     \widehat{\Delta \psi_{\sigma}}\\
     - \widehat{\Delta \varphi_{\sigma}}
   \end{array}\right) . \left(\begin{array}{c}
     1\\
     0
   \end{array}\right) \right) d s \]
and
\[ \tilde{S}_2 = \int_T^{T - t} e^{\lambda_2 (\xi) (T - t - s)} | \xi | \left(
   P^{- 1} \left(\begin{array}{c}
     \widehat{\Delta \psi_{\sigma}}\\
     - \widehat{\Delta \varphi_{\sigma}}
   \end{array}\right) . \left(\begin{array}{c}
     0\\
     1
   \end{array}\right) \right) d s. \]
Now, we define the norm
\[
\| \hat{f} \|_k \assign \left\| \frac{(1 + | \xi |)^k \hat{f}}{e^{- c | \xi
   | t} + e^{- c | \xi | (T - t)}} \right\|_{L^{\infty} (\mathbb{R}^2)}.
\]
the assumptions of Theorem \ref{aiaiaia}, then the same estimates hold for the norms $\ . \_k$ for any $k > 2$. It is noteworthy that the proofs of Theorem \ref{aiaiaia} were conducted using the norm $\ . \ _3$. We can verify that the same calculations are valid for the norms $\ . \ _k$ for any $k > 2$ (the constants will depend on the value of $k$). Consequently, if the initial data $\psi_0, \varphi_T$ meet the requirements of Theorem \ref{aiaiaia}, then the same estimates are applicable for the norms $\ . \ _k$ for any $k > 2$.
\[
\| (1 + | \xi |)^5(1+\sigma|\xi|) \hat{\psi}_0 \|_{L^{\infty} (\mathbb{R}^2)} + \| (1 + |
   \xi |)^5 (| \xi | + \sigma | \xi |^2) \hat{\varphi}_T \|_{L^{\infty}
   (\mathbb{R}^2)} \leqslant \varepsilon_0,
\]
with $\varepsilon_0$ small enough, then $\psi_{\sigma}, \varphi_{\sigma}$ the
solution of (\ref{mq1})-(\ref{mq2}) for this initial data satisfy
\[ \| (1+\sigma|\xi|)\widehat{\psi_{\sigma}} \|_5 + \| (| \xi | + \sigma | \xi |^2)
   \widehat{\varphi_{\sigma}} \|_5 \leqslant K \varepsilon_0 . \]
We now check that
\[ \| \tilde{S}_1 \|_3 + \| \tilde{S}_2 \|_3 \leqslant K (\|
   (1+\sigma|\xi|)\widehat{\psi_{\sigma}} \|_5 + \| (| \xi | + \sigma | \xi |^2)
   \widehat{\varphi_{\sigma}} \|_5) \leqslant K \varepsilon_0 . \]
From this and a fixed point argument, we
deduce that $\hat{u}, \hat{v}$ defined in (\ref{solee}) satisfy
\[ \| \hat{u} \|_3 + \| \hat{v} \|_3 \leqslant K \sigma \varepsilon_0, \]
concluding the proof.
\end{proof}
\subsection{Proof of Lemmas \ref{vortex} and \ref{light}}
\begin{proof}
  We look for a solution of
  \[ \left\{\begin{array}{l}
       \partial_t \psi = (f (\bar{\rho}) - 2 \bar{\rho}) \partial_x \psi +
       \bar{\rho} f^2 (\bar{\rho}) \Delta \varphi\\
       \partial_t \varphi = f (\bar{\rho}) \partial_x \varphi - \frac{1}{f
       (\bar{\rho})} \psi
     \end{array}\right. \]
  of the form
  \[ \psi = A \cos (a x + b y + c t), \]
  \[ \varphi = B \sin (a x + b y + c t) . \]
  Direct identification shows that this is a solution if
  \begin{eqnarray*}
    A (c - a (f (\bar{\rho}) - 2 \bar{\rho})) - B (a^2 + b^2) \bar{\rho} f^2
    (\bar{\rho})  & = & 0\\
    A \left( \frac{- 1}{f (\bar{\rho})} \right) + B (- c + f (\bar{\rho}) a) &
    = & 0.
  \end{eqnarray*}
  This system admits a non zero solution for $A, B$ if and only if
  \[ (c - a (f (\bar{\rho}) - 2 \bar{\rho})) (- c + f (\bar{\rho}) a) - (a^2 +
     b^2) \bar{\rho} f (\bar{\rho}) = 0, \]
  that we write as a polynom on $c$:
  \[ - c^2 + c (2 a (f (\bar{\rho}) - \bar{\rho})) - a^2 f (\bar{\rho}) (f
     (\bar{\rho}) - \bar{\rho}) - b^2 \bar{\rho} f (\bar{\rho}) = 0. \]
  This admits real solutions if
  \[ (2 a (f (\bar{\rho}) - \bar{\rho}))^2 - 4 (a^2 f (\bar{\rho}) (f
     (\bar{\rho}) - \bar{\rho}) + b^2 \bar{\rho} f (\bar{\rho})) \geqslant 0,
  \]
  in other words (dividing by $4 a^2$)
  \[ (f (\bar{\rho}) - \bar{\rho})^2 \geqslant \left( f (\bar{\rho}) (f
     (\bar{\rho}) - \bar{\rho}) + \frac{b^2}{a^2} \bar{\rho} f (\bar{\rho})
     \right) . \]
  We deduce that if $f (\bar{\rho}) - \bar{\rho} < 0$ (which is equivalent to
  $\bar{\rho} > \frac{\rho_m}{2}$), then we can find such a solution by taking
  $\frac{b^2}{a^2}$ small enough.
The proof of Lemma \ref{light} can be done similarly, and there we check that the condition for the existence of a non zero solution becomes
$ 4 (a^2 + b^2) \bar{\rho} f (\bar{\rho}) \geq 0$ instead, which is always satisfied.
\end{proof}

\section{Acknowledgments}
The work of the authors is supported by Tamkeen under the NYU Abu Dhabi Research Institute grant CG002 of the center SITE. Data sharing not applicable to this article as no datasets were generated or analysed during the current
study. The authors have no competing interests to declare that are relevant to the content of this article.

\bibliographystyle{unsrt}
\bibliography{Refs_GMP}

\begin{thebibliography}{10}

\bibitem{gm2023}
Mohamed Ghattassi and Nader Masmoudi.
\newblock Non-separable mean field games for pedestrian flow: Generalized
  hughes model.
\newblock {\em \text{https://arxiv.org/abs/2310.04702}}.

\bibitem{hughes2002continuum}
Roger~L Hughes.
\newblock A continuum theory for the flow of pedestrians.
\newblock {\em Transportation Research Part B: Methodological}, 36(6):507--535,
  2002.

\bibitem{burger2014mean}
Martin Burger, Marco Di~Francesco, Peter~A Markowich, and Marie-Therese
  Wolfram.
\newblock Mean field games with nonlinear mobilities in pedestrian dynamics.
\newblock {\em Discrete and Continuous Dynamical Systems-B}, 19(5):1311--1333,
  2014.

\bibitem{carrillo2016improved}
Jose~A Carrillo, Stephan Martin, and Marie-Therese Wolfram.
\newblock An improved version of the hughes model for pedestrian flow.
\newblock {\em Mathematical Models and Methods in Applied Sciences},
  26(04):671--697, 2016.

\bibitem{bruna2017cross}
Maria Bruna, Martin Burger, Helene Ranetbauer, and Marie-Therese Wolfram.
\newblock Cross-diffusion systems with excluded-volume effects and asymptotic
  gradient flow structures.
\newblock {\em Journal of Nonlinear Science}, 27:687--719, 2017.

\bibitem{amadori2023mathematical}
Debora Amadori, Boris Andreianov, Marco Di~Francesco, Simone Fagioli, Th{\'e}o
  Girard, Paola Goatin, Peter Markowich, Jan~F Pietschmann, Massimiliano~D
  Rosini, Giovanni Russo, et~al.
\newblock The mathematical theory of hughes' model: a survey of results.
\newblock {\em arXiv preprint arXiv:2305.10076}, 2023.

\bibitem{mouzouni2020quasi}
Charafeddine Mouzouni.
\newblock On quasi-stationary mean field games models.
\newblock {\em Applied Mathematics \& Optimization}, 81(3):655--684, 2020.

\bibitem{camilli2023quasi}
Fabio Camilli and Claudio Marchi.
\newblock On quasi-stationary mean field games of controls.
\newblock {\em Applied Mathematics \& Optimization}, 87(3):47, 2023.

\bibitem{Lions2023}
Pierre-Louis Lions.
\newblock Cours au coll\`ege de france. \textit{www.college-de-france.fr}.

\bibitem{porretta2015weak}
Alessio Porretta.
\newblock Weak solutions to fokker--planck equations and mean field games.
\newblock {\em Archive for Rational Mechanics and Analysis}, 216:1--62, 2015.

\bibitem{achdou2018mean}
Yves Achdou and Alessio Porretta.
\newblock Mean field games with congestion.
\newblock {\em Annales de l'Institut Henri Poincar{\'e} C, Analyse non
  lin{\'e}aire}, 35(2):443--480, 2018.

\bibitem{gomes2015time0}
Diogo~A Gomes, Edgard~A Pimentel, and H{\'e}ctor S{\'a}nchez-Morgado.
\newblock Time-dependent mean-field games in the subquadratic case.
\newblock {\em Communications in Partial Differential Equations}, 40(1):40--76,
  2015.

\bibitem{gomes2016time}
Diogo~A Gomes, Edgard Pimentel, and H{\'e}ctor S{\'a}nchez-Morgado.
\newblock Time-dependent mean-field games in the superquadratic case.
\newblock {\em ESAIM: Control, Optimisation and Calculus of Variations},
  22(2):562--580, 2016.

\bibitem{ambrose2018strong}
David~M Ambrose.
\newblock Strong solutions for time-dependent mean field games with
  non-separable hamiltonians.
\newblock {\em Journal de Math{\'e}matiques Pures et Appliqu{\'e}es},
  113:141--154, 2018.

\bibitem{cardaliaguet2012long}
Pierre Cardaliaguet, Jean-Michel Lasry, Pierre-Louis Lions, and Alessio
  Porretta.
\newblock Long time average of mean field games.
\newblock {\em Networks and heterogeneous media}, 7(2):279--301, 2012.

\bibitem{cardaliaguet2013long}
Pierre Cardaliaguet, J-M Lasry, P-L Lions, and Alessio Porretta.
\newblock Long time average of mean field games with a nonlocal coupling.
\newblock {\em SIAM Journal on Control and Optimization}, 51(5):3558--3591,
  2013.

\bibitem{gomes2010discrete}
Diogo~A Gomes, Joana Mohr, and Rafael~Rigao Souza.
\newblock Discrete time, finite state space mean field games.
\newblock {\em Journal de math{\'e}matiques pures et appliqu{\'e}es},
  93(3):308--328, 2010.

\bibitem{cardaliaguet2019long}
Pierre Cardaliaguet and Alessio Porretta.
\newblock Long time behavior of the master equation in mean field game theory.
\newblock {\em Analysis \& PDE}, 12(6):1397--1453, 2019.

\bibitem{porretta2018turnpike}
Alessio Porretta.
\newblock On the turnpike property for mean field games.
\newblock {\em Minimax Theory and its Applications}, 3(2):285--312, 2018.

\bibitem{cesaroni2021brake}
Annalisa Cesaroni and Marco Cirant.
\newblock Brake orbits and heteroclinic connections for first order mean field
  games.
\newblock {\em Transactions of the American Mathematical Society},
  374(7):5037--5070, 2021.

\bibitem{cirant2019existence}
Marco Cirant.
\newblock On the existence of oscillating solutions in non-monotone mean-field
  games.
\newblock {\em Journal of Differential Equations}, 266(12):8067--8093, 2019.

\bibitem{cirant2018variational}
Marco Cirant and Levon Nurbekyan.
\newblock The variational structure and time-periodic solutions for mean-field
  games systems.
\newblock {\em Minimax Theory and its Applications}, 3(2):227--260, 2018.

\bibitem{cirant2021long}
Marco Cirant and Alessio Porretta.
\newblock Long time behavior and turnpike solutions in mildly non-monotone mean
  field games.
\newblock {\em ESAIM: Control, Optimisation and Calculus of Variations}, 27:86,
  2021.

\bibitem{mimikos2022regularity}
Nikiforos Mimikos-Stamatopoulos and Sebastian Munoz.
\newblock Regularity and long time behavior of one-dimensional first-order mean
  field games and the planning problem.
\newblock {\em arXiv preprint arXiv:2204.06474}, 2022.

\bibitem{munoz2022classical}
Sebastian Mu{\~n}oz.
\newblock Classical and weak solutions to local first-order mean field games
  through elliptic regularity.
\newblock {\em Annales de l'Institut Henri Poincar{\'e} C}, 39(1):1--39, 2022.

\bibitem{munoz2023classical}
Sebastian Munoz.
\newblock Classical solutions to local first-order extended mean field games.
\newblock {\em ESAIM: Control, Optimisation and Calculus of Variations}, 29:14,
  2023.

\bibitem{cesaroni2023stationary}
Annalisa Cesaroni and Marco Cirant.
\newblock Stationary equilibria and their stability in a kuramoto mfg with
  strong interaction.
\newblock {\em arXiv preprint arXiv:2307.09305}, 2023.

\bibitem{carmona2022synchronization}
Rene Carmona, Quentin Cormier, and H~Mete Soner.
\newblock Synchronization in a kuramoto mean field game.
\newblock {\em arXiv preprint arXiv:2210.12912}, 2022.

\bibitem{lasry2007mean}
Jean-Michel Lasry and Pierre-Louis Lions.
\newblock Mean field games.
\newblock {\em Japanese journal of mathematics}, 2(1):229--260, 2007.

\bibitem{cirant2021maximal}
Marco Cirant and Alessandro Goffi.
\newblock Maximal l q-regularity for parabolic hamilton--jacobi equations and
  applications to mean field games.
\newblock {\em Annals of PDE}, 7(2):19, 2021.

\bibitem{cirant2022existence}
Marco Cirant and Daria Ghilli.
\newblock Existence and non-existence for time-dependent mean field games with
  strong aggregation.
\newblock {\em Mathematische Annalen}, 383(3-4):1285--1318, 2022.

\bibitem{cardaliaguet2019master}
Pierre Cardaliaguet, Fran{\c{c}}ois Delarue, Jean-Michel Lasry, and
  Pierre-Louis Lions.
\newblock {\em The master equation and the convergence problem in mean field
  games:(ams-201)}.
\newblock Princeton University Press, 2019.

\end{thebibliography}
\end{document}